\documentclass[letterpaper,11pt,reqno]{amsart}

\usepackage[utf8]{inputenc}
\usepackage{amssymb}
\usepackage{amsmath}
\usepackage{amsthm}
\usepackage{amsfonts}
\usepackage{bbm}
\usepackage{bookmark}
\usepackage{hyperref}
\hypersetup{pdfstartview={FitH}}
\usepackage{color}
\usepackage{mathrsfs} 
\usepackage[normalem]{ulem}

\theoremstyle{plain}
\newtheorem{theorem}{Theorem}[section]
\newtheorem{lemma}[theorem]{Lemma}

\newtheorem{proposition}[theorem]{Proposition}

\theoremstyle{definition}

\numberwithin{equation}{section}
\newtheorem*{theorem*}{Theorem} 

\newcommand{\Z}{{\mathbb Z}}
\newcommand{\R}{{\mathbb R}}

\newcommand{\C}{{\mathcal C}}

\newcommand{\T}{{\mathcal T}}

\newcommand{\Sp}{{\mathcal S}}
\newcommand{\I}{\mathcal{I}}

\def\det{\operatorname{det}}
\def\wt{\widetilde}

\DeclareFontFamily{U}{mathx}{\hyphenchar\font45}
\DeclareFontShape{U}{mathx}{m}{n}{
<5> <6> <7> <8> <9> <10>
<10.95> <12> <14.4> <17.28> <20.74> <24.88>
mathx10
}{}
\DeclareSymbolFont{mathx}{U}{mathx}{m}{n}
\DeclareFontSubstitution{U}{mathx}{m}{n}
\DeclareMathAccent{\widecheck}{0}{mathx}{"71}

\title{Local bounds for singular Brascamp-Lieb forms with cubical structure}
\author[P. Durcik]{Polona Durcik}
\address{Schmid College of Science and Technology, Chapman University, One University Drive, Orange, CA 92866, USA}
\email{durcik@chapman.edu}

\author[L. Slav\'ikov\'a]{Lenka Slav\'ikov\'a}
\address{Department of Mathematical Analysis, Faculty of Mathematics and Physics, Charles University, Sokolovsk\'a 83, 186 75 Praha 8, Czech Republic}
\email{slavikova@karlin.mff.cuni.cz}

\author[C. Thiele]{Christoph Thiele}
\address{Mathematisches Institut, Universit\"at Bonn, Endenicher Allee 60, 53115 Bonn, Germany}
 \email{thiele@math.uni-bonn.de}

\date{\today}

\begin{document}


\begin{abstract}
We prove a range of $L^p$ bounds
for singular Brascamp-Lieb forms with cubical structure. We pass through sparse and local bounds, the latter
proved by an iteration of Fourier expansion, telescoping, and the Cauchy-Schwarz inequality. We allow $2^{m-1}<p\le \infty$ with $m$ the dimension of the cube, extending an earlier result that required $p=2^m$.
The threshold $2^{m-1}$ is sharp in our theorems.
\end{abstract}

\maketitle

\section{Introduction}
 
A singular Brascamp-Lieb form with cubical structure is a principal value integral
\begin{equation}\label{def:sblc}
p.v. \int_{\R^{2m}} \prod_{j\in \C} F_j(\Pi_jx ) K(\Pi x) \, dx .
\end{equation}
It is called "Brascamp-Lieb", because the integrand is a product of several functions $F_j$ and $K$, each
composed with a linear surjection $\Pi_j$ or $\Pi$.  It is called singular, because $K$ is a Calder\'on-Zygmund kernel.

The cubical structure of \eqref{def:sblc} lies in the choice of surjections. We let $\C$ be the set of functions $$j:\{1,\dots, m\}\to \{0,1\} $$
and $\Pi_j$ the projection from $\R^{2m}$ to $\R^m$ given by
\begin{equation}\label{def:Pij}
\Pi_j(x^0_1,\dots, x^0_m,x^1_1,\dots ,x^1_m)^T:=(x_1^{j(1)},\dots, x_m^{j(m)})^T ,
\end{equation}
while the surjection $\Pi:\R^{2m}\to \R^m$ is generic. 
Each index $j$ can naturally be identified with the corner of an $m$-dimensional cube $[0,1]^m$. 
The cubical structure allows for a symmetrization procedure, pioneered in 
\cite{K12:tp}, of the tuple of functions $F_j$ along reflection symmetries of the cube. This is crucial for our understanding of $L^p$ estimates for these forms.
The paper \cite{KTZ-K} provides evidence that singular Brascamp-Lieb forms with cubical structure are key to understanding more general singular Brascamp-Lieb
forms. Singular Brascamp-Lieb forms with cubical structure have found applications in enumerative combinatorics \cite{DK20}, \cite{DKR18}, \cite{K20}
and ergodic theory \cite{DKST19}, \cite{KS20}.

We fix $m\ge 2$ throughout this paper and always denote by $\C$ the cube and by $\Pi_j$ the projections as in \eqref{def:Pij}. We also fix the notation that components of vectors in 
$\R^{2m}$ are written $x^l_i$ as in \eqref{def:Pij} and we use matrix product notation for linear maps as in \eqref{def:Pij}.

We assume in this paper, unless otherwise stated, that each $F_j$ is a real-valued smooth compactly supported function on $\R^m$. The principal value in \eqref{def:sblc} means that the Schwartz function 
$$\phi=\prod_{j\in \C} F_j\circ \Pi_j$$ is paired with the tempered distribution $K\circ \Pi$.
A Calder\'on-Zygmund kernel $K$ of degree $n$ in $\R^m$ is the Fourier transform of integration against a function $\widehat{K}$ on
$\R^m\setminus \{0\}$  which satisfies the symbol estimates \begin{equation*} 
|\partial^\alpha \widehat{K}(\eta)|\le |\eta|^{-|\alpha|}
\end{equation*}
for all multi-indices of order $|\alpha|\le n$.
Hence the principal value integral can be written as a standard integral 
$$p.v. \int _{\R^{2m}} \phi(x) K(\Pi x)\,dx =
\int_{\R^m} \widehat{\phi}(-\Pi^T \eta)  \widehat{K}(\eta) \, d\eta .$$

An $m\times m$ matrix $A$ is called $\Delta$-regular, if
$$\max(|A|,|\det A|^{-1})\le \Delta,$$
where $|A|$ denotes the operator norm of $A$. Note that one can also control $|\det A|$ and $|A^{-1}|$
of a $\Delta$-regular matrix by suitable $\Delta'$ depending on $\Delta$.

The main purpose of this paper is to prove the following theorem, which generalizes a theorem in \cite{DT20}.
\begin{theorem}
\label{mainp}
 Let $\Delta>1$. For each $j\in \C$ let
 \begin{equation}\label{E:prange}
    2^{m-1}<p_j\le \infty
\end{equation} 
and assume
\begin{equation}\label{holder}
    \sum_{j\in \C} \frac 1{p_j}=1.
\end{equation}
Then there is a constant $C$ such that the following holds.

 Let $\Pi$ be a real $m\times 2m$ matrix such that for each $j\in \C$ the composition $\Pi_j\Pi^T$ is $\Delta$-regular. 
For all
Calder\'on-Zygmund kernels $K$ in $\R^m$ of order $2^{6m}$  and all tuples $(F_j)_{j\in \C}$   we have 
$$\Big|p.v. \int_{\R^{2m}} \prod_{j\in \C} F_j(\Pi_jx ) K(\Pi x) \, dx \Big|\le C\prod_{j\in \C} \|F_j\|_{p_j}.$$
\end{theorem}
The result in \cite{DT20} is the special case $p_j=2^m$ for all $j\in \C$.  There the necessity of regularity of $\Pi_j\Pi^T$ has been discussed in detail. The exponent $2^m$ is very particular and allows for a global argument in \cite{DT20}. 
Any deviation from this exponent requires a localization of the arguments. 
The main point of the present paper is to develop such localization for these singular Brascamp-Lieb forms.
A key difficulty arises from the general position of $\Pi$, which prevents the use of algebraic telescoping identities and
instead requires cone decomposition \eqref{decrsg} of certain multipliers.
This leads to rectangles of arbitrary eccentricity, expressed via matrices $D$ or $E$, and 
correspondingly a use of  a mollified strong maximal function for example in  \eqref{e:refquestion},
\eqref{e:csum1}
and in 
\eqref{e:teles7}.
The localization procedure also
leads to boundary terms associated with arbitrary eccentricities, which require some additive measure theoretic arguments
in \eqref{e:set1} and subsequent displays. 
The use of mollified strong maximal function,
a superposition of anisotropically rescaled 
Hardy-Littlewood maximal functions
as seen in Lemma \ref{L:endpoint_bound}, and additive measure theory in the context of singular Brascamp-Lieb inequalities seem a
novelty in our paper.

Theorem \ref{mainp} should be compared with \cite{D15}, which proves  a special case with $m=2$ of the theorem, and with
\cite{St19}
in the simpler dyadic setting. 
Both of these references avoid general surjections $\Pi$.

On a formal level, we prove Theorem~\ref{mainp} with the technique of sparse bounds, which has seen
prominent developments in the past decade and is a vehicle towards 
further results such as weighted and vector valued inequalities. Concretely, Theorem \ref{mainp} follows from the sparse  bounds formulated in the next Theorem \ref{mains}. 

A dyadic interval is one of the form 
$[2^k n, 2^k(n+1))$ with integers $k,n$. A dyadic cube $Q$ in $\R^m$ is of the form $I_1\times \cdots \times I_m$, where the $I_i$ are dyadic intervals of the same length. We denote by $Q^*$ the cube with the same center as $Q$ and three times the side length.
A sparse ($\frac 12$-sparse) collection $\Sp$ of dyadic cubes is one that satisfies for all $Q\in \Sp$
$$\sum_{Q'\in \Sp: Q'\subseteq Q}|Q'|\le  2 |Q|,$$
with $|Q|$ denoting the volume of $Q$. 
For a measurable set $S \subseteq \R^m$, typically a cube or parallelepiped, and an integrable function $F$ on $\R^m$, we denote 
$$
[F]_S:=\Big(\frac{1}{|S|} \int_{S} |F|^{2^{m-1}}\Big)^{2^{1-m}}.
$$

For a $2m\times m$ matrix $J$ we denote the
upper $m\times m$ block by $J^0$ and the lower $m\times m$ block by $J^1$.
Assuming that $J^1$ is regular, we write
$$\Pi^J:=(I,-J^0(J^1)^{-1}),$$
which implies $\Pi^J\circ J=0$.
We will see in Section \ref{S:maintheorem} that all maps $\Pi$ in Theorem~\ref{mainp} may and will be assumed to be of the form $\Pi^J$. 

\begin{theorem}\label{mains}
 Let $\Delta>1$. 
For each $j\in \C$, let $p_j$ be such that 
\eqref{E:prange} and \eqref{holder} hold.
Then there is a constant $C$ such that the following holds.

Let $J$ be a real $2m\times m$ matrix such that for each $j\in \C$ the matrix 
\begin{equation*}
    A_j:=\Pi_jJ
\end{equation*}
is $\Delta$-regular.
For all
Calder\'on-Zygmund kernels $K$ in $\R^m$ of order $2^{6m}$  and all tuples $(F_j)_{j\in \C}$   there is a
sparse  collection $\Sp$ of dyadic cubes in $\R^m$ such that
\begin{equation}\label{E:sparse}
\Big| p.v. \int_{\R^{2m}} \Big( \prod_{j\in \C} F_j(\Pi_j x)\Big) K(\Pi^J x)\,dx\Big|\leq C \sum_{Q\in \Sp} |Q| \prod_{j\in \C} [F_j\circ A_j ]_{Q^*}.
\end{equation}
\end{theorem}

The reduction of Theorem \ref{mainp} to Theorem \ref{mains} will be presented in Section \ref{S:maintheorem}. It follows from a well known paradigm   used in \cite{CD-PO18} in the context of the bilinear Hilbert transform, another prominent singular Brascamp--Lieb form.
In our instance, the reduction is essentially a consequence of  Theorem  1.11 in 
\cite{Z-K19}.   Indeed, one
obtains in \cite{Z-K19} more general bounds in weighted $L^p$ spaces, with good dependence
on the weight classes. Further, one can deduce vector valued bounds as in \cite{Ni19}.

We highlight a further reduction 
of Theorem \ref{mains}
to local bounds formulated in Theorem \ref{maint} below. The reduction follows some standard procedures and will be elaborated in Section \ref{S:reduction}. Note
the appearance of the modified strong maximal function in
 Theorem \ref{maint}, which will be key to the inductive approach to this theorem.

The quantity $K(\Pi^J x)$ is invariant under translation in direction of the kernel of $\Pi^J$, which is the range of $J$, and obeys strong regularity assumption invariant under isotropic scaling. It therefore can be naturally written as an integral over the group spanned by
dilations and by translations in direction of the range of $J$ of forms defined by certain bump functions in $x$ in place of $K(\Pi^J x)$.
Following \cite{DT20}, using suitable windowed Fourier transforms on the multiplier $\widehat{K}$, we write \eqref{def:sblc} as superposition of forms
\begin{equation}\label{form:decomposed}
\int_0^\infty  \int_{\R^m} \frac{c(t)}{1+|u|^{4m}}\int_{\R^{2m}}  \Big ( \prod_{j\in \C} F_j(\Pi_j x)\Big )  (\partial_i^0 \partial_{i}^1g)_t(x-Jp + ut)  \, dx \, dp \, \frac{dt}{t},
\end{equation}
where  $1\le i\le m$, $\|c\|_\infty\le 1$, $u\in \R^{2m}$,  $g(x)=e^{-\pi|x|^2}$ is the Gaussian in $\R^{2m}$, and for a function $\phi$ on $\R^m$ we use the notation $\phi_t(x) = t^{-m} \phi (t^{-1}x)$.  
We have written $\partial_i^0$
and $\partial_i^1$ in analogy to the convention in \eqref{def:Pij} for vectors in $\R^{2m}$.
It suffices to prove bounds of Theorem \ref{mains} with
the left-hand side of \eqref{E:sparse}
replaced by the absolute value of 
\eqref{form:decomposed}.

The variables $(p,t)$ in \eqref{form:decomposed} 
are in the upper half space $\R^m \times(0,\infty) $, the above mentioned symmetry group. We decompose the upper
half space into boxes 
$Q\times  ({\ell(Q)}/{2},\ell(Q))$, where $Q$ is a dyadic cube and $\ell(Q)$ denotes its side length. For a collection $\T$ of dyadic cubes we define
$$\Omega_\mathcal{T} := \cup_{Q\in \T} Q\times  ({\ell(Q)}/{2},\ell(Q))$$
and we set
\begin{equation}
\label{form:local}
\Lambda_{\T}((F_j)_{j\in \C}):= \int_{\Omega_{\mathcal{T}}} \frac{c(t)}{1+|u|^{4m}}
   \int_{\R^{2m}}   \Big ( \prod_{j\in \C} F_j(\Pi_j x)\Big )
 (\partial_i^0 \partial_{i}^1g)_t(x-Jp+ut)  \, dx \, dp \, \frac{dt}{t}.
\end{equation}

A collection $\T$ of dyadic cubes  is called a convex tree or a stopping time,  if there exists $Q_\T\in \T$ such that for all $Q\in \T$ we have $Q\subseteq Q_\T$ and 
$Q'\in \T$ whenever 
 $Q\subseteq Q'\subseteq Q_\T$.

Let $\mathcal{R}$ be the set of rectangular boxes in $\R^m$, that is boxes of the form
$R=I_1\times \dots \times I_m$, with intervals $I_1,\dots ,I_m$ that are not necessarily
dyadic and not necessarily of equal length.
We denote by  $\varepsilon_R$  the reciprocal of  the eccentricity of $R$, that is  
$$
\varepsilon_R=\big( \min_{1\leq k \leq m} |I_k|\big ) \big( \max_{1\leq k \leq m} |I_k| \big)^{-1}.
$$
We define for $\alpha\ge 0$ the modified strong
maximal operator
\begin{equation}\label{E:tildestrong}
    {M}_\alpha(F)(x):= \sup_{R\in \mathcal{R}: x\in R}(\varepsilon_{R}^{\alpha})
    ^{2^{1-m}}
    [F]_{R}.
\end{equation}
This is a weighted version, with weight depending on the eccentricity, of the classical strong $L^{2^{m-1}}$ maximal operator.
We will typically see  expressions
of the type ${M}_\alpha(F)(Ax)$
for some $m\times m$ matrix $A$. 
For a dyadic cube $Q$ we define 
$${M}_\alpha(F)(AQ):=
 \sup_{R\in \mathcal{R}: AQ\subseteq R}
 (\varepsilon_{R}^{\alpha})
    ^{2^{1-m}}
    [F]_{R}.$$
Here $AS$ for some set $S$ means the set of all $Ax$ with $x\in S$.
For a collection $\T$ of dyadic cubes
we define 
$${M}_\alpha(F)(A\T):=
 \sup_{Q\in \T} {M}_\alpha(F)(AQ).$$

\begin{theorem}
\label{maint}
Let $\Delta>1$ and  $\alpha<\frac 1m$.
There exists  $C>0 $
such that the following holds.
Let $J$, $A_j$ be as in Theorem \ref{mains}.
 For    any convex tree $\mathcal{T}$,
for all $c,u,i$ satisfying conditions specified near \eqref{form:decomposed} and all tuples    $(F_j)_{j\in \mathcal{C}}$ we have
\[| \Lambda_{\mathcal{T}}((F_j)_{j\in \C}) | \leq C  |Q_{\T}| \prod_{j\in \C} {M}_\alpha(F_j)(A_j\T). \]
\end{theorem}

The upper bound on the parameter $\alpha$ is not so important for the deduction of Theorem \ref{mains} from Theorem \ref{maint}, here a classical $2^{m-1}$-Hardy-Littlewood operator would do, which is a limit of $M_\alpha$ as $\alpha$ tends to $\infty$. The upper bound is however needed for the proof of Theorem \ref{maint},
except in particular cases 
such as  $m=2$ or  $J^0=J^1=I$.
While the classical strong $2^{m-1}$-maximal operator, which is the case $\alpha=0$ of $M_\alpha$,
does not satisfy a weak type $2^{m-1}$ bound,
the operator $M_\alpha$ with $\alpha>0$
does, as we will show in Lemma \ref{L:endpoint_bound}. Related to this we also recall that the classical strong maximal operator cannot be sparsely dominated~\cite{BCOR}, while the result of Theorem~\ref{maint} with $\alpha>0$ allows us to obtain the sparse bounds of Theorem~\ref{mains}.

We highlight a further reduction of Theorem \ref{maint} to estimates in Theorem \ref{mainx}
for a form that has localization 
of the $x$ variables rather than localization in the $(p,t)$ variables. The localization in $x$ is somewhat more precise, because
it truncates $2m$ variables rather than $m$ variables. Moreover, it facilitates Fourier
transform techniques in the  variable $p$.
A truncation in $x$ is however
not useful in the deduction of Theorem \ref{mains}, as this truncation  requires global information on the tree and therefore poorly commutes with the tree selection algorithm in the proof of sparse domination.

With notation near \eqref{form:local}, define
$\wt{\Lambda}_{\T}((F_j)_{j\in \C})$ as
\begin{equation*}
\sum_{k\in \Z} \int_{2^{k-1}}^{2^k} \int_{\R^m} \frac{c(t)}{1+|u|^{4m}}
   \int_{\R^{2m}}   \Big ( \prod_{j\in \C} F_j1_{A_jT_k}(\Pi_j x)\Big ) (\partial^0_i \partial^1_{i}g)_t(x-Jp+ut) \,  dx \, dp \, \frac{dt}{t},
\end{equation*}
where for a tree $\mathcal{T}$ we have written  
$$  T_k := \cup\{Q\in \T: l(Q)=2^k\} .$$

\begin{theorem}\label{mainx}

Let $\Delta>1$ and  $\alpha<\frac 1m$.
There exists  $C>0 $
such that the following holds.
Let $J$, $A_j$ be as in Theorem \ref{mains} .
 For    any convex tree $\mathcal{T}$,
for all $c,u,i$ satisfying conditions specified near \eqref{form:decomposed}
, and all tuples   $(F_j)_{j\in \C}$ we have 
\[| \wt{\Lambda}_{\mathcal{T}}((F_j)_{j\in \C}) | \leq C  |Q_{\T}| \prod_{j\in \C} {M}_\alpha(F_j)(A_j\T). \]
\end{theorem}

The following theorem shows sharpness of one assumption in Theorem \ref{mainp}:
\begin{theorem}\label{T:sharp}
The variant of Theorem \ref{mainp},
where the threshold $2^{m-1}$ in \eqref{E:prange}
is replaced by a smaller number, is false.
\end{theorem}

The examples used to prove Theorem \ref{T:sharp} are similar to the ones in \cite{K12:tp} and \cite{DR21}.
Theorem \ref{T:sharp} also implies that similarly the lower bound $2^{m-1}$ in Theorem \ref{mains} can not be lowered.
We stress that we do not know the sharp
range of tuples of exponents $(p_j)_j$ for the conclusion of Theorem \ref{mainp} to hold. It may be possible to lower some of the exponents $p_j$ below $2^{m-1}$ at the expense of demanding additional constraints on the other exponents.

The results in this paper are continuous analogues of
some of the estimates in \cite{St19} in the dyadic 
setting. It would be desirable to have a more complete understanding of the analogs of results in \cite{St19} in the continuous setting.
In the broader context, there are a number of similar results in the dyadic setting which await a transfer into the continuous setting, such as for example \cite{MT17}, \cite{OT11}, \cite{KTZ-K}.

In Section \ref{S:maintheorem}, we reduce Theorem \ref{mainp} to Theorem \ref{mains}.
In Section \ref{S:reduction}, we reduce Theorem \ref{mains} to Theorem \ref{maint}.
In Section \ref{S:morered}, we reduce Theorem \ref{maint} to Theorem \ref{mainx}, this section should be compared with ~\cite[Lemma 6 and Lemma 7]{D15}.
In Section \ref{S:general}, we prove
Theorem \ref{mainx}
by an induction on the number of symmetries of the form
and the tuple $(F_j)_{j\in \C}$ as in \cite{DT20}. The elaborate induction statement splits
the proof into three natural steps, which are
done in three separate subsections.
In Section \ref{S:sharpness}, we prove Theorem \ref{T:sharp}.
 
A survey on singular Brascamp-Lieb forms is found in \cite{DT19}. We like to point the reader also at some more recent developments since the survey
such as \cite{bm}, \cite {mz}.
Comparing for example the results in \cite{bm} with our result, note that there the argument of the kernel $K$ has much larger dimension than the arguments of the functions, while in our result the argument of the kernel has same dimension as the arguments of the functions. Also \cite{bm}
has more general exponents than H\"older. It might be interesting to study Brascamp-Lieb data in the vicinity of our data that lead to non-H\"older scaling. We have not endeavoured in this direction yet.

 \subsection*{Acknowledgements}  The first author is supported by NSF DMS-2154356.
 The second author is supported by the Primus research programme PRIMUS/21/SCI/002 of Charles University. The third author acknowledges support
by the Deutsche Forschungsgemeinschaft (DFG, German Research Foundation) under Germany's Excellence Strategy - EXC-2047/1 - 390685813 as well as SFB 1060. Part of the research was carried out during a delightful workshop
on Real and Harmonic Analysis
at the Oberwolfach Research Institute for Mathematics.
The authors thank the anonymous referee for their careful reading of the paper and  valuable suggestions.

\section{Proof of Theorem \ref{mainp}}

\label{S:maintheorem}

In this section we reduce Theorem \ref{mainp} to Theorem \ref{mains}. 
Let $\Delta>1$. Let $(p_j)_{j\in \C}$ satisfy \eqref{E:prange} and \eqref{holder}. Let $K$ be a
Calder\'on-Zygmund kernel in $\R^m$ of order $2^{6m}$ and  let  $(F_j)_{j\in \C}$ be a tuple of smooth compactly supported functions.   Let $\Pi$ be a real $m\times 2m$ matrix such that $\Pi_j\Pi^T$ is $\Delta$-regular for each $j\in \C$.

 The $\Delta$-regularity of $\Pi_0\Pi^T$ and $\Pi_1\Pi^T$ implies that the left and right $m\times m$ blocks of $\Pi$ are $\Delta$-regular. Here we denote by $0$ and $1$ the respective constant functions in $\C$. 
    Thus,  the matrix $\Pi$ is of the form $(({J}^0)^{-1},-(J^1)^{-1})$ for some  $m\times m$ $\Delta$-regular matrices $J^0, J^1$.
    We write
$$K(\Pi x)=K(J_0^{-1}J_0\Pi x)=\tilde{K}(J_0 \Pi x)
=\tilde{K}(\Pi^J x)
$$
with $\tilde{K}=K\circ J_0^{-1}$
and $J=((J^0)^T,(J^1)^T)^T$. Then $\tilde{K}$
is up to a constant multiple depending on $\Delta$ also Calder\'on-Zygmund, and
we will apply Theorem \ref{mains} with the kernel $\tilde{K}$
and the matrix $J$.
 
 We verify that the matrices 
 $A_j=\Pi_j J$  are  $\Delta'$ regular for all $j$ and a $\Delta'$ depending on $\Delta$.
 The upper bound on the norm of $A_j$ follows 
 from upper bounds on the norm
 of $J$, and it then suffices to prove a lower bound on the
 absolute value of the determinant of $A_j$.
 By $\Delta$-regularity of $J_1$, it suffices to prove lower bounds on the determinant
 of 
 $$B_j:=A_j((J^1)^{-1})=
 \Pi_j((J^0(J^1)^{-1})^T, I)^T.$$
 
 The matrix $B_j$ 
 arises from 
 $V:=J^0(J^{1})^{-1}$
 by replacing each row with an 
 index $i$ with $j(i)=1$
 by the corresponding
 $i$-th row of the identity matrix.
 The determinant of $B_j$ is therefore equal to the determinant of the submatrix $W$ of $V$ obtained by erasing all
 rows and columns with index $i$ such that $j(i)=1$.
Similarly, the determinant of $-W$ is 
equal to the determinant of
$\Pi_{\overline{j}}(\Pi^J)^T$,
 where $\overline{j}$ denotes the corner of $\C$ opposite to $j$.
 The latter determinant is however bounded below by the assumed
 $\Delta$-regularity of 
 $\Pi_{\overline{j}}(\Pi^J)^T$. This proves the desired lower bound on
 the determinant of $A_j$.

Theorem \ref{mains} yields a 
sparse  collection $\Sp$ of dyadic cubes in $\R^m$ such that the analogue of \eqref{E:sparse} for $\tilde{K}$ holds. 
To deduce Theorem \ref{mainp} from this bound, we first estimate the right-hand side of \eqref{E:sparse} by a sum of averages over cubes in  dyadic grids. Here, by a dyadic grid we mean any collection $\mathcal{D}$ of measurable subsets of $\R^m$ such that any   $Q,Q'\in \mathcal{D}$ are either disjoint or one is contained in the other.   It follows from Definition 2.1 and Theorem 3.1 in \cite{LN19}  that there exist $3^m$ dyadic grids $\mathcal{D}_i$, $1\leq i \leq 3^m$, such that each $Q^*$ associated with a cube $Q \in \Sp$ belongs to $\mathcal{D}_i$ for some $i$. Thus,   
\begin{equation}\label{sparsegrd}
\sum_{Q \in \Sp} |Q| \prod_{j\in \C} [F_j\circ A_j ]_{Q^*} \leq \sum_{i=1}^{3^m}\sum_{Q^* \in \mathcal{S}_i} |Q^*| \prod_{j\in \C} [F_j\circ A_j ]_{Q^*},
\end{equation}
where
\[\mathcal{S}_i=\{Q^*\in \mathcal{D}_i : Q\in \mathcal{S} \} . \]
The collections $\mathcal{S}_i$ are $\frac{1}{2}3^{-m}$-sparse because
\[\sum_{{(Q')}^{*}\in \mathcal{S}_i: {(Q')}^*\subseteq Q^*}|{Q'}^*|\leq 3^m \sum_{{Q'}\in \mathcal{S}: {(Q')}^*\subseteq Q^*} |{Q'}|  \leq 2\cdot 3^m|Q^*| , \]
where in the last inequality we used that  $(Q')^*\subseteq Q^*$ implies that $Q'$ is contained in one of the $3^m$ cubes of the same size as $Q$ that are contained in $Q^*$ and we applied the assumption of sparsity of $\mathcal S$ on each of these cubes.
Applying
Theorem  1.11 in 
\cite{Z-K19}  with $\rho_i:=0$, $r_i:=1$,  $t_i:=p_i$,  and $w_i:=1$ to \eqref{sparsegrd} for each $1\leq i \leq 3^m$  yields   Theorem \ref{mainp}.

Alternatively, the reduction of Theorem \ref{mainp} to Theorem \ref{mains}  might possibly be accomplished with estimates for sparse operators rather than sparse forms as in \cite{LMS14}. We did not try to pursue this approach but
refer to \cite{LN19} for a survey.
In the formalism of outer $L^p$ spaces of \cite{DoT15}, the reduction is
an estimate of the
right-hand side of \eqref{E:sparse}  with a suitable outer H\"older inequality by
$$\leq C\|1_{\mathcal S}\|_{L^\infty(\ell^1)}\prod_{j\in \C} \|\mathcal{F}_j F_j\|_{L^{p_j}(\ell^\infty)}$$
for suitable embeddings $\mathcal{F}_j$.
Sparsity of $\mathcal S$ is tautological to a bound 
on $\|1_{\mathcal S}\|_{L^\infty(\ell^1)}$, and embedding theorems give the estimates 
$ \|\mathcal {F}_j F_j\|_{L^{p_j}(\ell^\infty)}\le C\|F_j\|_{p_j}$.

\section{Proof of Theorem \ref{mains}}

\label{S:reduction}

We present the proof of Theorem \ref{mains}, using Theorem \ref{maint} and some lemmas that will be stated and proved at the end of this section.
Let  $\Delta$,  $p_j$ be  as in Theorem \ref{mains}.
We need to find a constant
$C$ such that the sparse bound of Theorem \ref{mains} holds for all
$J$, $A_j$  as in Theorem \ref{mains}, all
tuples $(F_j)_{j\in \C}$ of compactly supported smooth functions and all Calder\'on-Zygmund kernels $K$ in $\R^m$.
We will find constants for various related inequalities 
and by abuse of notation will use the letter $C$ for each of them. As a consequence, the meaning of $C$ may change from line to line. 

As discussed in the introduction and elaborated in detail in \cite{DT20}, 
by superposition it suffices to show that the sparse bound holds for the specific class of Calder\'on-Zygmund kernels associated with forms $\Lambda$ of the type \eqref{form:decomposed}.
The integral in the $(t,p)$ variables in \eqref{form:decomposed}
is absolutely integrable,
a typical theme in Calder\'on Zygmund theory that we elaborate
in Lemma \ref{l:integrable}.
It therefore suffices to show the sparse bound 
for the form $\Lambda_{{\mathcal Q}'}$
as in $\eqref{form:local}$
with arbitrary finite collection ${\mathcal Q}'$ of dyadic cubes.
Let thus $c,u,i$ be given as detailed near \eqref{form:decomposed}
and let ${\mathcal Q}'$ 
be given.
By making ${\mathcal Q}'$
larger if needed, we may assume
that ${\mathcal Q}'$ is convex in the sense that if $Q\subseteq Q'\subseteq Q''$ and $Q,Q''\in {\mathcal Q}'$, then also $Q'\in {\mathcal Q}'$. We may also assume
 that each maximal cube $Q$ in ${\mathcal Q'}$ 
 has the property that the support of each of  the functions $F_j\circ A_j$
 is contained in $Q^*$.
Here and in what follows, maximality of cubes is with respect
to the partial order by set inclusion.

We construct a sparse collection.
For each dyadic cube $Q$
let $\mathcal I(Q)$ be the set of maximal dyadic cubes $Q'\subseteq Q$ satisfying
\begin{equation*}
\sup_{R\in {\mathcal R}:A_jQ'\subseteq R} (\varepsilon_R^\alpha)^{2^{1-m}} [F_j 1_{A_jQ^*}]_R>c[F_j\circ A_j]_{Q^*} 
\end{equation*}
for some $j\in \C$, 
where $c$ is a certain positive constant  determined a few lines below.
For each $Q'\in \mathcal I(Q)$, there is $j\in \C$ such that $$M_\alpha (F_j1_{A_jQ^*})(A_j x) > c[F_j\circ A_j]_{Q^*}$$ for all $x\in {Q'}.$
With the weak bound for $M_\alpha$ in Lemma~\ref{L:endpoint_bound} below, we conclude
\begin{equation}\label{E:hardy_littlewood}
\sum_{Q' \in \mathcal I(Q)} |Q'| 
\leq C\sum_{j\in \C} |\{x\in \R^m:{M}_\alpha (F_j 1_{A_jQ^*})(A_jx)>c[F_j\circ A_j]_{Q^*}\}|\end{equation}
$$\leq C\sum_{j\in \C} (c[F_j\circ A_j]_{Q^*})^{-2^{m-1}} \int_{A_jQ^*} |F_j|^{2^{m-1}}
=Cc^{-2^{m-1}} \sum_{j\in \C}|A_jQ^*| \le  \frac{1}{2}|Q|,
$$
where $c$ is chosen large enough depending on $\Delta$ such that the last inequality holds.

Let $\Sp_0$ be the collection
of maximal dyadic cubes 
in ${\mathcal Q}'$. It is a finite collection of disjoint cubes.
For $n\geq 1$ define
$$
\Sp_n:=\bigcup_{{Q} \in \Sp_{n-1}} \I({Q})$$
and let $\Sp$ be the union of the sets $\Sp_n$ over all natural numbers $n$.
Using \eqref{E:hardy_littlewood} and summing a geometric series, we deduce that the collection $\Sp$ is sparse.

We aim to show
\begin{equation} \label{e:aimtoshow}
|\Lambda_{{\mathcal Q}'}((F_j)_{j\in \C})|
\leq C \sum_{Q\in \Sp} |Q| \prod_{j\in \C} [F_j\circ A_j ]_{Q^*}.
\end{equation}
Given a dyadic cube $Q$, denote by $\T_{\leq Q}$ the set of all dyadic cubes $Q'$ in  ${\mathcal Q}'$
with $Q'\subseteq Q$.
We observe
\begin{equation}\label{e:recbegin}\Lambda_{{\mathcal Q}'}
((F_j)_{j\in \C})=
\sum_{Q\in \Sp_0}\Lambda_{\T_{\le Q}}((F_j)_{j\in \C})=
\sum_{Q\in \Sp_0}\Lambda_{\T_{\le Q}}((F_j1_{A_jQ^*})_{j\in \C}).
\end{equation}
The first equality follows from the definition of $\Sp_0$ and $\T_{\le Q}$. The second equality follows by the assumption on the maximal cubes   $Q$ in ${\mathcal Q}'$.

Denote for any cube $Q\in \mathcal{S}$ by $\T_Q$ the convex tree 
$$
\T_Q :=\T_{\leq Q} \setminus \bigcup_{Q' \in \I(Q)} \T_{\leq {Q'}}.
$$
We have for $Q\in \mathcal{S}$
\begin{equation}\label{E:split}
\Lambda_{\T_{\leq Q}}((F_j 1_{A_jQ^*})_{j\in \C})
=\Lambda_{\T_Q}((F_j 1_{A_jQ^*})_{j\in \C})
+\sum_{Q'\in \I(Q)} \Lambda_{\T_{\leq Q'}} ((F_j 1_{A_jQ^*})_{j\in \C}).
\end{equation}
By Theorem \ref{maint}
and the definition of
$\mathcal{I}(Q)$, the first  term on the right-hand side of \eqref{E:split} can be estimated as
$$|\Lambda_{\T_Q}((F_j1_{A_jQ^*})_{j\in \C})|
\le C|Q| \prod_{j\in \C} M_\alpha(F_j1_{A_jQ^*})(A_j \T_Q) 
\le C |Q|\prod_{j\in \C} [F_j\circ A_j]_{Q^*}.$$
To estimate  the second term on the right-hand side of \eqref{E:split}, fix $Q' \in \I(Q)$ and split
$$
F_j1_{A_jQ^*}=F_j 1_{A_j(Q')^*} + F_j {1_{A_jQ^* \setminus A_j(Q')^*}}
$$
for each  $j\in \C$.
Multilinearity of the form
$\Lambda_{\T_{\leq Q'}}$ gives
$$\Lambda_{\T_{\leq Q'}} ((F_j 1_{A_jQ^*})_{j\in \C})=
\Lambda_{\T_{\leq Q'}} ((F_j 1_{A_j(Q')^*})_{j\in \C})
+\sum_{(E_j)_{j\in \C}\in {\mathcal E}}
\Lambda_{\T_{\leq Q'}} ((F_jE_j)_{j\in \C}),
$$
where ${\mathcal E}$ is the set of tuples with entries 
$1_{A_j(Q')^*}$ or  $1_{A_jQ^* \setminus A_j(Q')^*}$ such that
at least one $j$ satisfies $E_j=1_{A_jQ^* \setminus A_j(Q')^*}$.
The first term on the right hand side will be estimated by a recursive procedure, while the second term  is estimated by Lemma~\ref{L:maximal_function} and by the choice of ${\mathcal I}(Q)$. We obtain for the second term
$$
\sum_{(E_j)_{j\in \C}\in {\mathcal E}}
\Lambda_{\T_{\leq Q'}} ((F_jE_j)_{j\in \C})\le
C |Q'| \prod_{j\in \C} M_\alpha(F_j1_{A_jQ^*})(A_j Q')\le 
C |Q'| \prod_{j\in \C} [F_j\circ A_j]_{Q^*}.$$

Combining \eqref{e:recbegin}
and  \eqref{E:split} and
using the bounds we have just discussed, we have
\begin{equation*}
|\Lambda_{{\mathcal Q}'}((F_j)_{j\in \C})|
\le C \sum_{Q\in \Sp_0} |Q| \prod_{j\in \C} [F_j\circ A_j]_{Q^*}
+\sum_{Q\in \Sp_{1}} |\Lambda_{\T_{\leq Q}} ((F_j 1_{A_jQ^*})_{j\in \C})|.
\end{equation*}
Iterating this argument over $\Sp_1$, $\Sp_2$, and so on, and using that $\Sp_n\cap {\mathcal Q}'$ is empty for large $n$ by finiteness of ${\mathcal Q}'$, we obtain \eqref{e:aimtoshow}.
This completes the proof of Theorem \ref{mains}.

\begin{lemma}
[Weak type $2^{m-1}$ bound for $M_\alpha$]
\label{L:endpoint_bound}
For $\alpha>0$ and
$m\ge 2$, there is a $C>0$ such that the following holds. For
a measurable function ${F}$  on $\R^m$ and $\lambda>0$, we have with 
${M}_\alpha$ as in
\eqref{E:tildestrong},
\begin{equation*}
|\{x :M_\alpha ({F})(x)>\lambda\}| \le C \lambda^{-2^{m-1}} 
\|F\|_{2^{m-1}}^{2^{m-1}}.
\end{equation*}
\end{lemma}
\begin{proof}
We denote by $\mathcal Z$ the set consisting of all $k=(k_1,\dots,k_{m})\in \Z^{m}$ with $k_m=0$.
Given $k\in \mathcal Z$, we define 
$
{\mathcal R}_{k}$
to be the set of all 
$I_1 \times \dots \times I_{m} \in \mathcal{R}$ such that
for each $1\leq i \leq m-1$,
$$2^{k_i-1} <\frac{|I_i|}{|I_{m}|} \leq 2^{k_i} .$$ 
We introduce the maximal function
$$
M_{k} F(x) := \sup_{ R \in {\mathcal R}_{k}: x\in R}[F]_R .
$$
We claim the weak type bound 
$$|\{x:M_{k} ({F})(x)>\lambda\}| \le C \lambda^{-2^{m-1}} \|F\|_{2^{m-1}}^{2^{m-1}}
.$$
By a  change of variables
$y=Dx$ if necessary, with diagonal matrix $D$ with entries $2^{k_i-k_m}$, it suffices to show this claim for
$k_1=\dots=k_{m-1}=0$.
As each rectangular box in ${\mathcal R}_{(0,\dots,0)}$ has eccentricity very close to $1$, it is contained in a cubical box of volume no larger than
$2^m$ times the volume of the rectangular box. The claim follows then from the weak type $1$ bound for the standard Hardy-Littlewood
maximal operator, applied to the function $|F|^{2^{m-1}}$.

Define for $k\in \mathcal Z$ 
$$\varepsilon_k:=\frac{\min_{1\le i\le m} 2^{k_i}} {\max_{1\le i\le m} 2^{k_i}}$$
and note that for $R\in {\mathcal R}_k$ we have
$${2^{-1}\varepsilon _k \le \varepsilon_R \le 2 \varepsilon_k.}$$
We also have  
\begin{equation*}
\sum_{k\in \mathcal Z}
\varepsilon_k^{\alpha} <C ,
\end{equation*}
because the exponential decay
of the summand in the length of the vector $k$ beats the polynomial growth of the volume of the balls in $\Z^{m-1}$.
We obtain
$$|\{y:{M}_\alpha ({F})(y)>\lambda\}| 
\le |\{y:\sup_{k\in \mathcal Z} (\varepsilon_k^\alpha)^{2^{1-m}} {M}_k ({F})(y)>C \lambda\}|$$
$$
\le  \sum_{k\in \mathcal Z}
|\{y:(\varepsilon_k^{\alpha})^{2^{1-m}}  {M}_k ({F})(y)>C \lambda \}| $$
$$\le \sum_{k\in \mathcal Z}
C
\varepsilon_k^\alpha
 \lambda ^{-2^{m-1}} \|F\|_{2^{m-1}}^{2^{m-1}}
 \le 
C \lambda ^{-2^{m-1}} \|F\|_{2^{m-1}}^{2^{m-1}}
.$$
\end{proof}

\begin{lemma}[Brascamp-Lieb inequality for $\Pi_j$]
\label{l:brascamplieb}
For any   measurable functions 
$F_j$ on $\R^m$ we have 
\begin{equation*}
\Big|\int_{\R^{2m}} \prod_{j\in \C} F_j(\Pi_j x)\,dx\Big|
\leq  \prod_{j\in \C} \|F_j\|_{2^{m-1}}.
\end{equation*}
\end{lemma}
\begin{proof}
We induct on $m\geq 1$.  By Fubini, we have
$$\int_{\R^2} |F_0(x^0_1)F_1(x^1_1)| \, dx=\|F_0\|_1\|F_1\|_1,$$
which gives the desired inequality for
$m=1$. For $m>1$, we proceed by Cauchy-Schwarz:
$$
\int_{\R^{2m}} \prod_{j\in \C} |F_j(\Pi_j x)|\,dx=
\int_{\R^{2m}} \prod_{\sigma=0}^1   \prod_{j(m)=\sigma} |F_j(\Pi_j x)|  \,dx
$$
$$
\le
\int_{\R^2}
\prod_{\sigma=0}^1\Big(\int_{\R^{2m-2}} \prod_{j(m)=\sigma} |F_j(\Pi_j x)|^2 \, dx^0_1\dots dx^0_{m-1} \, dx^1_{1}\ldots dx^1_{m-1}\Big)^{1/2} \, dx^0_m\, dx^1_{m}.
$$
Note that the integrals in $x^0_m$ and $x^1_{m}$ split.
By induction, we obtain  an upper bound by
$$
\prod_{\sigma=0}^1\int_{\R}\prod_{j(m)=\sigma} \|F_j^2\|_{L^{2^{m-2}}(x_1,\dots ,x_{m-1})}^{1/2} \, dx_m \leq C \prod_{\sigma=0}^1 \prod_{j(m)=\sigma} \|F_j\|_{2^{m-1}} = C \prod_{j\in \mathcal{C}} \|F_j\|_{2^{m-1}} ,
$$
where we have used the identity $\|F^2\|_{2^{m-2}}=\|F\|_{2^{m-1}}^2$ and H\"older's inequality in the integral over $x_m$.
\end{proof}

\begin{lemma}[Integrability over upper half space]
\label{l:integrable}
The integral in the $(t,p)$ variables in \eqref{form:decomposed} is absolutely integrable for  compactly supported smooth functions $F_j$.
\end{lemma}

\begin{proof}
For fixed $t$, the integral 
\begin{equation}\label{e:fixt}\int_{\R^{m}}\Big | \int_{\R^{2m}}  \Big ( \prod_{j\in \C} F_j(\Pi_j x)\Big )  (\partial^0_i \partial^1_{i}g)_t(x-Jp + ut) \,  dx \Big | \, dp  \end{equation}
can be estimated with Fubini and  Lemma \ref{l:brascamplieb} by
$$\leq C \Big (\prod_{j\in \C} \|F_j\|_{2^{m-1}}\Big )
\sup_{x\in \R^{2m}}\int_{\R^m} |(\partial^0_i \partial^1_{i}g)_t(x-Jp + ut) |\, dp .$$
The  supremum in $x$ is bounded by $Ct^{-m}$,
which makes \eqref{e:fixt} absolutely integrable over the half line $t>1$ with measure $dt/t$.
For $t<1$, a partial integration on the $x$ integral estimates 
\eqref{e:fixt}
by
$$C t^2\int_{\R^m}\int_{\R^{2m}}  |\partial^0_i \partial^1_{i} \Big ( \prod_{j\in \C} F_j(\Pi_j x)\Big ) | g_t(x-Jp + ut)  \, dx \, dp.$$
We split $ \R^{2m}$
as the direct sum of the kernel $V$ of $\Pi^J$ and its complement. We then estimate
\eqref{e:fixt} with Fubini by
$$ t^2\int_{V} \int_{V^\perp} |\partial^0_i \partial^1_{i} \Big ( \prod_{j\in \C} F_j(\Pi_j (y+z))\Big )  | \int_{\R^m}g_t(y+z-Jp + ut)\, dp \, dy \, dz$$
$$\le  t^2\int_{V}  \sup_{y\in V^\perp}  |\partial^0_i \partial^1_{i} \Big ( \prod_{j\in \C} F_j(\Pi_j (y+z))\Big )  | \Big (\int_{y\in V^\perp}\int_{\R^m}g_t(y+z-Jp + ut) \, dp \, dy \Big )\, dz.$$
The Gaussian integral is bounded by a constant independent of $t$, and the remaining integral is finite since $V^\perp$ is orthogonal to $J\R^m$. Therefore,   the whole last display is bounded by $Ct^2$. This can be integrated over $t<1$ against the measure $dt/t$.
\end{proof}

\begin{lemma}[Gaussian estimate]
\label{l.gauss}
With notation as in Theorem \ref{maint},
 for every $x\in \R^{2m}$, $t>0$ and $p,p'\in \R^m$  with $|p'- p|<2^{m+3}t$, we have
\begin{equation}\label{e:ge}
\Big|\frac{1}{1+|u|^{4m}}
 (\partial^0_i \partial^1_{i}g)_t(x-Jp+ut)\Big|\le Ct^{-2m} \sum_{n\ge 0} 2^{-4nm}\chi(|x-Jp'|\le 2^{n}t)
 \end{equation}
 with the indicator $\chi(\mathcal{A})$ that is $1$ if $\mathcal{A}$ is satisfied and $0$ otherwise.
\end{lemma}

\begin{proof}
Assume first the case
\begin{equation}\label{e:sxjp}
t^{-1} |x-Jp'|\le 8 +2^{m+5}|J|
\end{equation}
with the operator norm $|J|$ of $J$. We estimate the left-hand side of 
\eqref{e:ge} by the quantity
$t^{-2m}\|\partial^0_i\partial^1_{i}g\|_\infty$,
which for suitable $C$ is bounded by
the single term on the right-hand side 
with the smallest index $n$, for which    the right-hand side of \eqref{e:sxjp} is less than $2^n$. 

It remains to consider the case that $\eqref{e:sxjp}$ is not satisfied. Pick $n\ge 4$
such that
$$2^{n-1}< t^{-1} |x-Jp'|\le 2^n .$$
Considering only the $n$-th term on the right-hand side of \eqref{e:ge}, it  suffices to show
\begin{equation}\label{e:lxjp}\Big|\frac{1}{1+|u|^{4m}}
 \partial^0_i \partial^1_{i}g(t^{-1}(x-Jp)+u)\Big|\le C2^{-4nm}. 
 \end{equation}
If $|u|>2^{n-3}$, this follows from a bound for $\|\partial^0_i\partial^1_{i}g\|_\infty$. If $|u|\le 2^{n-3}$, we have
$$|t^{-1}(x-Jp)+u|\ge
|t^{-1}(x-Jp')|-2^{m+3}|J|-2^{n-3}\ge 2^{n-3},$$
where we used the assumed bound on $p-p'$ and estimated  $|J|$ by the assumed reverse inequality to \eqref{e:sxjp}.
Estimate \eqref{e:lxjp}  then follows from the Schwartzian bound,
$$| \partial^0_i\partial^1_{i}g(y)|\le C|y|^{-4m}$$
for $|y|\ge 1$.
 \end{proof}

\begin{lemma}[Off-diagonal estimate for a tree]\label{L:maximal_function}
Let  $\Delta>1$. There is a constant $C>0$ such that for
 $J$, $A_j$ as in Theorem~\ref{mains},
for any convex tree $\T$
and any tuple $(F_j)_{j\in \C}$ of compactly supported bounded measurable  functions on $\R^m$ such that 
for some $j_0\in \C$, the support of $F_{j_0}\circ A_{j_0}$ is
disjoint from $Q_\T^*$, we have
\begin{equation*}
|\Lambda_\T((F_j)_{j\in \C})| \le C  |Q_\T| \prod_{j\in \C} \sup_{Q\in {\mathcal Q}: Q_\T^*\subseteq Q^*} [F_j\circ A_j]_{Q^*}.
\end{equation*}
\end{lemma}
Here $\mathcal Q$ denotes the set of all dyadic cubes
in $\R^m$.

\begin{proof}
Let $Q'\in \T$. Using Lemmas \ref{l.gauss} and   \ref{l:brascamplieb} we estimate
 \begin{equation}\label{e:indq} \Big|\int_{l(Q')/2}^{l(Q')}
 \int_{Q'}
  \frac{c(t)}{1+|u|^{4m}}
   \int_{\R^{2m}}   \Big ( \prod_{j\in \C} F_j(\Pi_j x)\Big )
 (\partial^0_i \partial^1_{i}g)_t(x-Jp+ut)  \, dx \, dp \, \frac{dt}{t} \Big |
 \end{equation}
 $$ \le C|Q'|\sup_{l(Q')/2\le t\le l(Q')}   \int_{\R^{2m}}   \Big ( \prod_{j\in \C} |F_j(\Pi_j x)|\Big )
 t^{-2m} 
\sum_{n\ge 0} 2^{-4nm}\chi(|x-J{\mathfrak c}(Q')|\le 2^{n}t) \,
dx$$
$$ \le C  \sum_{n\ge 0} 2^{-4nm} |Q'|^{-1}\int_{\R^{2m}}    \Big ( \prod_{j\in \C} |F_j(\Pi_j x)|
\chi(|\Pi_j x-A_j \mathfrak c(Q')| \le 2^n{l(Q')} ) \Big ) \,
dx$$
$$ \le C  \sum_{n\ge 0} 2^{-4nm} |Q'|^{-1}\prod_{j\in \C} 
\|F_j(y)\chi(|y-A_j {\mathfrak c}(Q')|\le 2^n{l(Q')})
\|_{L^{2^{m-1}}(y)}.
$$
Here ${\mathfrak c}(Q)$ denotes the center
of the cube $Q$.

For the summation indices $n$ with 
$|A_{j_0}^{-1}| 2^nl(Q')<l(Q_\T)$, we see that the factor with $F_{j_0}$ vanishes because  $F_{j_0}\circ A_{j_0}$ is supported outside $Q_\T^*$ while $\mathfrak c(Q')$ is inside $Q_\T$.
Thus, it remains to sum over $n$ with $|A_{j_0}^{-1}| 2^nl(Q')\geq l(Q_\T)$. We estimate
$$\|F_j(y)
\chi(|y-A_j \mathfrak c(Q')|\le 2^n{l(Q')})
\|_{L^{2^{m-1}}(y)}$$
$$\le C2^{nm2^{1-m}}|Q'|^{2^{1-m}}
\sup_{Q\in {\mathcal Q}: Q^*_\T\subseteq Q^*} [F_j\circ A_j]_{Q^*}.
$$
In the last inequality  we used that $\mathfrak c(Q')\in Q_\T$ and 
that averages are over boxes of volume comparable or larger in volume than $Q_\T^*$. 

We continue the estimate of \eqref{e:indq} with 
$$ \le C  \sum_{|A_{j_0}^{-1}| 2^nl(Q')\ge l(Q_\T)} 2^{-4nm} 2^{2nm} |Q'|\prod_{j\in \C}  \sup_{Q\in {\mathcal Q}: Q_\T^*\subseteq Q^*} [F_j\circ A_j]_{Q^*}
$$
$$ \le C |Q'| \frac{|Q'|}{|Q_\T|} \prod_{j\in \C} \sup_{Q\in {\mathcal Q}: Q_\T^*\subseteq Q^*} [F_j\circ A_j]_{Q^*}.
$$
Summing \eqref{e:indq} over all $Q'\in \T$ and using 
$$ \sum_{Q'\in \T}|Q'| \Big(\frac{|Q'|}{|Q_\T|}\Big)  = \sum_{k\leq 0}2^{mk} \sum_{Q'\in \T: l(Q')=2^{k}l(Q_{\T})}|Q'|
\le \sum_{k\leq 0}2^{mk}|Q_\T| \leq C |Q_\T|
$$
proves the lemma.
\end{proof}

\section{Proof of Theorem \ref{maint}}

\label{S:morered}

We present the proof of Theorem 
\ref{maint} using Theorem \ref{mainx}, several lemmas from the previous section and one from the end of this section.
With notation as in Theorems \ref{maint} and 
\ref{mainx}, we
need to prove 
\begin{equation}\label{e:lambdadiff}
| \wt{\Lambda}_{\mathcal{T}}((F_j)_{j\in \C}) - {\Lambda}_{\mathcal{T}}((F_j)_{j\in \C}) |
 \leq C  |Q_{\T}| \prod_{j\in \C} {M}_\alpha(F_j)(A_j\T).
\end{equation}
We estimate the left-hand side of \eqref{e:lambdadiff}
with Lemma \ref{l.gauss} by
\begin{equation}\label{E:error1}
\sum_{n\ge 0}2^{-4mn}\sum_{k\in \Z} \int_{2^{k-1}}^{2^k} \int_{T_k}\int_{S_k^c} \Big( \prod_{j\in \C} |F_j(\Pi_j x)| \Big)
t^{-2m} \chi(|x-Jp|\le 2^{n}t)
\,dx\,dp\,\frac{dt}{t}
\end{equation}
\begin{equation}\label{E:error2}+ 
\sum_{n\ge 0}2^{-4mn}\sum_{k\in \Z} \int_{2^{k-1}}^{2^k} \int_{T_k^c}\int_{S_k} \Big( \prod_{j\in \C} |F_j(\Pi_j x)| \Big)
t^{-2m} \chi(|x-Jp|\le 2^{n}t)
\,dx\,dp\,\frac{dt}{t},
\end{equation}
where $S_k$ is the set of all $x\in \R^{2m}$ such that $\Pi_jx\in A_j T_k$ for all $j\in \C$.

Fix $n$ and $k$ and first estimate the corresponding summand in \eqref{E:error1} by

$$ C 2^{-2mk} \int_{T_k}\int_{S_k^c} \Big( \prod_{j\in \C} |F_j(\Pi_j x)| 
 \chi(|\Pi_j x-A_j p|\le 2^{n}2^k)\Big)
\,dx\,dp.$$
Let $E$ be the set of $p\in T_k$ such that the inner integral of the last display is not zero. Then we estimate the last display with the Brascamp-Lieb inequality from Lemma~\ref{l:brascamplieb} by
$$\le C 2^{-2mk} \int_{ E}  \prod_{j\in \C} \|F_j( y) 
 \chi(|y-A_j p|\le 2^{n}2^k)\|_{L^{2^{m-1}}(y)}
\,dp$$
\begin{equation}\label{e:error1int}
\le C 2^{2mn} |E| \prod_{j\in \C} 
{M}_\alpha(F_j)(A_j\T).
\end{equation}
Here the last inequality followed by identifying the 
$L^{2^{m-1}}$ norm of the truncated function $F_j$
as an average over a cube  of side length $2^{n}2^k$ about the point $A_jp$, which is contained in a set $A_jQ^*$
of comparable volume for a $Q$ containing a cube $Q'\in \T$.

We estimate $|E|$. If $p\in E$, then there is $x\in S_k^c$ such that for all $j\in \C$ we have
$$|\Pi_j x-A_j p|\le 2^{n}2^k.$$
By definition of $S_k$, there is a $j_0\in \C$ such that
$\Pi_{j_0} x\not \in A_{j_0} T_k$.
Let $Q_p$ be a dyadic cube of side length  $2^k$ containing $p$ and let $Q_x$ be a dyadic cube of side length $2^k$ such that
$\Pi_{j_0} x \in A_{j_0} Q_x$.
Then $Q_p\subseteq T_k$ and $Q_x \not \subseteq T_k$.
But both $A_{j_0}Q_p$ and $A_{j_0}Q_x$ are contained in the ball $B$ of radius
$C2^n2^k$ about $A_{j_0}p$ for 
sufficiently large $C>1$ depending on $\Delta$.
Hence there is
$$y\in \partial T_k \cap (2^k \Z)^m$$
such that $y\in A_{j_0}^{-1} B$. But then
$p$ is contained in the ball of radius $C2^n2^k$ about $y$. Hence
$$|E|\le C2^{nm} 2^{km} \#(\partial T_k \cap (2^k \Z)^m).$$
Using this estimate, 
\eqref{e:error1int}, and summing the geometric
series in $n$, we obtain
for \eqref{E:error1} the upper bound
$$\leq C\Big( \sum_{k \in \Z}2^{km} 
\#(\partial T_k \cap (2^k \Z)^m) \Big )\prod_{j\in \C} 
{M}_\alpha(F_j)(A_j\T).$$
With Lemma \ref{L:boundary}
we obtain the desired bound for \eqref{E:error1}.

It remains to estimate \eqref{E:error2}.
Fix again  $n$ and $k$ and estimate the corresponding summand by
$$ C 2^{-2mk} \int_{T_k^c}\int_{S_k} \Big( \prod_{j\in \C} |F_j(\Pi_j x)| 
 \chi(|\Pi_j x-A_j p|\le 2^{n}2^k)\Big)
\,dx\,dp.$$
Let $E$ be the set of $p\in T_k^c$ such that the inner integral of the last display is not zero. Then we estimate the last display with the Brascamp-Lieb inequality from Lemma \ref{l:brascamplieb} by
\begin{equation}
    \label{error2bl}
    \le C 2^{-2mk} \int_{ E}  \prod_{j\in \C} \|F_j( y) 
 \chi(|y-A_j p|\le 2^{n}2^k)\|_{L^{2^{m-1}}(y)}
\,dp.
\end{equation}
If $p\in E$, then there is $x\in S_k$ such that for all $j\in \C$ we have
$$|\Pi_j x-A_j p|\le 2^{n}2^k.$$
By definition of $S_k$, for every $j\in \C$ there is a $p_j\in T_k$ such that
$\Pi_jx=A_jp_j$.
Using the triangle inequality, we may estimate  \eqref{error2bl} by
$$
    \le C 2^{-2mk} \int_{ E}  \prod_{j\in \C} \|F_j( y) 
 \chi(|y-A_j p_j|\le 2^{n+1}2^k)\|_{L^{2^{m-1}}(y)}
\,dp
$$
\begin{equation*}
\le C 2^{2mn} |E| \prod_{j\in \C} 
{M}_\alpha(F_j)(A_j\T).
\end{equation*}
Here we may argue in the last inequality as for \eqref{e:error1int}, because $p_j\in T_k$.
We may also argue as before that
the ball of radius $C2^{n+1} 2^k$ about $p$ contains $p_j$ and we may argue as before that it also contains a point 
in $\partial T_k \cap (2^k \Z)^m$.
We may estimate as before
$$|E|\le C2^{nm} 2^{km} \#(\partial T_k \cap (2^k \Z)^m)$$
and conclude the desired bound for 
\eqref{E:error2}.
This concludes the proof of Theorem \ref{maint}.

\begin{lemma}\label{L:boundary}
There is a constant $C>0$ such that for any convex tree $\T$ in $\R^m$, 
$$
\sum_{k\in \Z} 2^{mk} \# (\partial T_k \cap (2^k \Z)^m) \le C |Q_\T|.
$$
\end{lemma}
\begin{proof}
Let   $\mathcal{Q}$ be the collection of maximal cubes $Q\subseteq Q^*_\T$ such that $Q\not \in \T$.
As these cubes are pairwise disjoint and contained in $Q^*_\T$, we have
$$\sum_{Q\in {\mathcal Q}}|Q|\le C |Q_\T|.$$
For each $Q\in {\mathcal Q}$ let
${\mathcal Q}_Q$ be the set of all cubes
$Q'\subseteq Q$ such that the boundary of $Q'$ intersects the boundary of $Q$. 
With $k$ such that $l(Q)=2^k$, we have
$$\sum_{Q'\in  {\mathcal Q}_Q}|Q'|
\le
\sum_{k'\le k} \sum_{Q'\in  {\mathcal Q}_Q: l(Q')=2^{k'}}|Q'|\le C\sum_{k'\le k}2^{k'-k}
|Q|\le C|Q|,$$
  where we counted the cubes $Q'$ of fixed side length by using that they are contained in a thin layer around the boundary of the cube $Q$.

For each point $p\in \partial T_k \cap (2^k \Z)^m$, there is a cube $Q_{p,k}$ of side length $2^k$ which is adjacent to $p$
and element of ${\mathcal Q}_Q$ for some
$Q\in \mathcal{Q}$. Conversely, each such cube
$Q_{p,k}$ is equal to some other cube $Q_{p',k'}$ only for $k'=k$ and  boundedly many points  
$p'\in \partial T_k \cap (2^k \Z)^m$.
We obtain
$$\sum_{k\in \Z} 2^{mk} \# (\partial T_k \cap (2^k \Z)^m) \le 
C\sum_{Q\in {\mathcal Q}}\sum_{Q'\in {\mathcal Q}_Q} |Q'|
\le 
C\sum_{Q\in {\mathcal Q}} |Q|\le 
C |Q_\T|.$$
\end{proof}

\section{Proof of Theorem \ref{mainx}}
\label{S:general}
The cube $\C$ has natural reflection symmetries.
We formulate in
Proposition \ref{p:induction} a statement, which can be proven by downward induction on a parameter $l$ counting the reflection symmetries that preserve the value of $\Lambda((F_j)_{j\in \C})$ for a given form $\Lambda$ and given functions $F_j$. This proposition, applied to the case of no symmetry assumption, implies Theorem \ref{mainx}.

The induction consists of three by now well established key arguments. Symmetrization of the form is achieved by the Cauchy-Schwarz inequality
as from \eqref{beforecs} to \eqref{aftercs0}. One breaks the cube into two halves  along
a reflection symmetry, and 
Cauchy-Schwarz symmetrizes the two pieces.

It is important that the precious cancellation of the Calder\'on-Zygmund kernel $K$ 
is preserved through this Cauchy-Schwarz by being moved to bump functions associated with the two faces that are not cut by the symmetry plane. 
This is second crucial ingredient, often done by partial integration or algebra, moving the cancellation around the faces.
Here, due to the general projection $\Pi$, one has to work with cone decompositions \eqref{decrsg} in addition to algebra \eqref{e:lastdisplay}
and these decompositions require careful control of the geometry and dependence of various parameters, done mostly in 
Section \ref{S:reexpansion}.

The third ingredient, to avoid that pieces already symmetrized will reappear in later steps, is a positivity argument. All the previously and the newly symmetrized pieces appear as positive terms \eqref{e:positivity} of a sum  and it suffices to estimate the sum rather than the individual pieces.

Define for a convex tree $\T$, for an injection $J:\R^{m}\to \R^{2m}$, and for a function 
$$K:(0,\infty)\times \R^{2m} \to {\mathbb C}$$ 
that is measurable in the first variable $t\in(0,\infty)$ 
such that $x\to t^{m}K(t,tx)$ is
Schwartz with integral zero and
Schwartz norm bounds uniform in $t$,  the form
\begin{equation}\label{deflambda4}
\Lambda(\T,J,K)((F_j)_{j\in \C})=
\sum_{k\in \Z}   \int_{\R^{2m}}  \Big(\prod_{j\in \C} (F_j 1_{(A_jT_k)})(\Pi_jx) \Big) \int_{2^{k-1}}^{2^k}
K(t, x) \,
\frac{dt}{t} \,dx.
\end{equation}
If $1\leq i \leq m$ and $j\in \C$, we define the reflected corner $i\ast j \in \C$ by determining that  $(i\ast j)(k)$ coincides 
with $j(k)$ at all positions $k$ except at the position $k=i$, where it differs.
Define for a function $\varphi:\R^{2m}\to \mathbb{C}$
and a regular $2m\times 2m$ matrix 
$D$
$$
\varphi_{D}(x)=(\det D)^{-1}  \varphi(D^{-1}x).
$$
We write $J^\sigma_h$ for the $h$-th row of the matrix $J^\sigma$.
\begin{proposition}\label{p:induction}
Let  $\Delta>1$ and $\alpha< \frac 1m$.
There exists  $C>0 $
such that the following holds.
Let $0\le l\le m$ be an integer.
Let
 $J$, $A_j$
be as in Theorem \ref{mains}
and write $\Pi:=\Pi^J$.
Assume
that for  each $1\le h\le l$ we have  $J^0_h=J^1_h$.
Let $\mathcal{T}$ 
be any convex tree.
Let  $D$ be a $2m\times 2m$ diagonal matrix whose diagonal entries $d_h^\sigma$ are positive,
satisfy $d_{h}^0=d_h^1=:d_h$, and, in case $l<m$, for all
$1\le h\le l< i\le m$ 
\begin{equation}\label{dassumption}
d_h\le d_{l+1}=d_i.
\end{equation}
Then, for every tuple $(F_j)_{j\in \C}$  satisfying $F_j=F_{i\ast j}$ for all $j\in \C$ and $1\leq i \leq l$,
\begin{equation} \label{e:indbound}
|\Lambda(\T,J,K)((F_j)_{j\in \C})|\le C
\Big(\prod_{i=1}^m \max\{d_i^{2\alpha},d_i^{-2}\}\Big)  |Q_{\T}| \prod_{j\in \C} M_\alpha(F_j)(A_j\T)
\end{equation}
holds for the following three types of kernels $K$.
If $i>l$, then  \eqref{e:indbound} holds for the kernel
\begin{equation}\label{e:firstkernel}
  K(t, x)= \frac{c(t)}{1+|u|^{4m}} \int_{\R^m}    (\partial^0_i \partial^1_{i} {g})_{tD}(x+Jp+tDu)\,dp
\end{equation}
with $c,u,i$   as in \eqref{form:local}. 
If $l>0$, then \eqref{e:indbound} holds for the kernel 
\begin{equation}\label{e:secondkernel}
  K(t,x) =\int_{\R^m}  |tD\Pi^T\tilde{\Theta}\xi|^2  g(tD\Pi^T\xi){e^{2\pi i \xi^T  \Pi x}}  \, d\xi, 
\end{equation}
and for the kernel
\begin{equation}\label{e:thirdkernel}
  K(t,x)=  \sum_{i=1}^l \int_{\R^m}    (\partial_i^0 \partial_{i}^1 g)_{tD}(x+Jp)\,dp.
\end{equation}
Here 
$\tilde{\Theta}$ 
is the 
diagonal $m\times m$ matrix with 
$i$-th diagonal entry $\tilde{\Theta}_{i}=0 $ if $i\le l$ and
$\tilde{\Theta}_{i}=1 $ if $i> l$ .
\end{proposition}

The bound \eqref{e:indbound} for 
\eqref{e:firstkernel}
in case $l=0$ and $D$ the identity matrix gives Theorem \ref{mainx}.
Note a sign change in front of $Jp$, which is harmless as we integrate over $\R^m$.
The bound \eqref{e:indbound} for 
\eqref{e:firstkernel}
in case $l=m$ is void
as there is no index $l<i \leq m$,
and hence the statement is vacuously true for $l=m$.

In the next three sections, we will show that 
bounds for \eqref{e:secondkernel}
for some $l$ follow from bounds for \eqref{e:firstkernel} for the same $l$,
that 
bounds for \eqref{e:thirdkernel}
for some $l$ follow from bounds for \eqref{e:secondkernel} for the same $l$,
and that 
bounds for \eqref{e:firstkernel}
for some $l$ follow from bounds \eqref{e:thirdkernel} for  $l+1$.
By induction, this concludes the proof of Theorem \ref{mainx}.
As the induction consists of finitely many steps, it is permissible that
the constants $C$
which may depend on $\Delta$, $m$, $\alpha$ but not on $D, \T,c,u,i$ may change each inductive step between hypothesis and conclusion.

\subsection{Expansion: Bounds for the first kernel  imply bounds for the second kernel}

\label{S:reexpansion}

Let $J$, $\alpha$, $0<l\le m$, $D$ be given, set $\Pi=\Pi^J$,  and 
let $K$ be the kernel 
 \eqref{e:secondkernel}.
If $l=m$, then  $\tilde{\Theta} =0$ and the bounds \eqref{e:indbound}
are trivial. Thus, we assume $l<m$.
Let $d:=d_{l+1}$ denote the maximal entry of $D$.
Let ${\Theta}:\R^{2m}\to \R^{2m}$ be the diagonal matrix consisting of two copies of $\tilde{\Theta}$
on the diagonal. 

\begin{lemma}\label{l:firstlemma}
For each $j\in \C$, the map 
$\Pi^T$ maps $(1-\tilde{\Theta})\R^m$ to
$(1-\Theta) \R^{2m}$ and the map
$\Pi_j \Pi^T$ maps $(1-\tilde{\Theta})\R^m$
onto itself. Finally, we have $\Theta  \Pi^T \tilde{\Theta}=\Theta \Pi^T$.

\end{lemma}
\begin{proof}
By assumption we have
$(1-\tilde{\Theta})J^0=(1-\tilde{\Theta})J^1$
and hence
$$(1-\tilde{\Theta})(-J^0(J^1)^{-1})=\tilde{\Theta}-1.$$
We obtain
$$(1-\tilde{\Theta})\Pi=(1-\tilde{\Theta})(I,-J^0(J^1)^{-1})=(1-\tilde{\Theta})(I,-I).$$
This proves the first statement by transposition.
Moreover,
\begin{equation}\label{e:ej}
(1-\tilde{\Theta})\Pi\Pi_j^T=(1-\tilde{\Theta})E_j
\end{equation}
for the diagonal matrix $E_j$ with $h$-th diagonal entry $(-1)^{j(h)}$. Hence $\Pi_j\Pi^T$ maps 
$(1-\tilde{\Theta})\R^m$
into itself. As $\Pi_j\Pi^T$
is regular, it maps $(1-\tilde{\Theta})\R^m$
onto itself. This proves the second statement. The last statement follows from
$$\Theta\Pi^T (1-\tilde{\Theta})=0,$$ 
which is clear from the first statement.
\end{proof}

\begin{lemma}\label{l:seclemma}
There is a constant $C$ depending on $\Delta$ such that for all $j\in \C$  and all $\xi\in \R^m$ we have
\begin{equation}\label{e:seclemma}
|\Pi^T \tilde{\Theta} \xi| 
\le C|\tilde{\Theta}\Pi_j\Pi^T\tilde{\Theta}\xi|
\le C|{\Theta}\Pi^T\tilde{\Theta}\xi|.
\end{equation}
\end{lemma}
\begin{proof}
The second inequality follows from $\tilde{\Theta}\Pi_j=\Pi_j \Theta$ and the fact that $\Pi_j$ is a projection. To prove the first inequality,
by homogeneity, we may assume  $|\tilde{\Theta}\xi|= 1$.
As $\Pi_j\Pi^T$ maps $(1-\tilde{\Theta})\R^m$ onto itself, there is an $\eta\in (1-\tilde{\Theta})\R^m$ such that 
$$\tilde{\Theta}\Pi_j\Pi^T\tilde{\Theta}\xi =\Pi_j\Pi^T(\tilde{\Theta}\xi+\eta).$$
It follows by $\Delta$-regularity of $\Pi_j\Pi^T$ that with $C$ depending on $\Delta$ 
\begin{equation}
    \label{thetanonzero}
|\tilde{\Theta}\Pi_j\Pi^T\tilde{\Theta}\xi|\ge C|\tilde{\Theta}\xi+\eta| \ge C|\tilde{\Theta}\xi|.
\end{equation}
The first inequality in \eqref{e:seclemma} follows by $\Delta$-regularity of both blocks of $\Pi^T$.
\end{proof}

Define $W=D\Pi^T\tilde{\Theta}\R^m$.
\begin{lemma}\label{l:thirdlemma}
There exists $\delta>0$ depending on  $\Delta$ such that for each $\eta\in W$ with $1\le |\eta|$  there is an $l<i\le m$ such that for all $\zeta\in W$ with $|\zeta-\eta|<\delta$ we have
$\delta\le |\zeta_i^0|$ and $\delta\le |\zeta_i^1|$.
\end{lemma}

\begin{proof}
With $C$ from Lemma \ref{l:seclemma}, we set $\delta=(2Cm)^{-1}$.
We then have   for any $\xi$ and  $j$
\begin{equation}\label{e:longthetas}
|D \Pi^T \tilde{\Theta} \xi| \le d| \Pi^T \tilde{\Theta} \xi| 
\le Cd|\tilde{\Theta}\Pi_j\Pi^T\tilde{\Theta}\xi|
=C|\tilde{\Theta}\Pi_jD\Pi^T\tilde{\Theta}\xi|,
\end{equation}
where we used 
$\tilde{\Theta}\Pi_j=\Pi_j \Theta$
and ${\Theta}D=d{\Theta}$ from \eqref{dassumption}.
Now fix $\eta\in W$ with $1\le |\eta|$. Then there is $\xi\in \R^m$ with  $\eta=D\Pi^T \tilde{\Theta} \xi$.
Pick $j$ so that for
$l<i\le m$ we have
$$|\eta_i^{j(i)}|=\min(
|\eta_i^0|,|\eta_i^1|).$$
Then, with \eqref{e:longthetas},
$$1\le |\eta|\le C |\tilde{\Theta}\Pi_j\eta|
\le C m\max_{l<i\le m} 
\min(
|\eta_i^0|,|\eta_i^1|).$$
Pick $l<i\le m$ so that the maximum is attained in the last display.
For each $\zeta\in W$ with
$|\zeta-\eta|<\delta$ we have  
$$(Cm)^{-1}\le \min(
|\eta_i^0|,|\eta_i^1|)\le \delta+ \min(
|\zeta_i^0|,|\zeta_i^1|).$$
Therefore, with the definition of
$\delta$,
$$\delta\le  \min(
|\zeta_i^0|,|\zeta_i^1|).$$
\end{proof}

We need a partition of unity of $W\setminus \{0\}$, which we find as follows.
Consider the shell
$$S:=\{\eta\in W : 1\le |\eta|\le 4\}.$$
Let $\Gamma$ be a 
maximal set of $\delta/4$-separated vectors in $S$. The cardinality of $\Gamma$ depends on $\Delta$.
For each $\gamma\in \Gamma$,
let $B_\gamma$ be the open ball in $W$ of radius  $\delta$ about
$\gamma$, and let $\rho_\gamma$ be
a smooth function supported on the closure of $B_\gamma$ and strictly positive on
$B_\gamma$.
The collection of balls $B_\gamma$  covers $S$, while the collection of all $2^sS$ with $s\in \Z$ covers $W\setminus \{0\}$.
Hence
$$\psi_\gamma(\eta):=
\rho_\gamma(\eta) (\sum_{s\in \Z}\sum_{\gamma\in \Gamma} \rho_\gamma(2^{s}\eta))^{-1}$$
is well defined, supported on the closure of $B_\gamma$, and smooth for each $\gamma\in \Gamma$. We have
$$\sum_{s\in \Z}\sum_{\gamma\in \Gamma} \psi_\gamma(2^{s}\eta )=1$$
for all $\eta \in W\setminus \{0\}$.
For each $\gamma \in \Gamma$, by construction of the ball $B_\gamma$ and Lemma \ref{l:thirdlemma},  there is an $i=i_\gamma$ with
$l<i\leq m$
and a smooth compactly supported function $q_\gamma$ with support and derivative bounds independent of $D$, so that for
$\eta$ in $B_\gamma$ we have
\begin{equation}\label{decomph}
|\eta|^2=
\eta^0_i\eta^1_{i} 
q_\gamma (\eta).
\end{equation}
We fix such a choice of $i_\gamma$ and $q_\gamma$.

We further need an annular  partition of unity of $\R^{2m}\setminus \{0\}$. 
Let $\rho$ be a non-negative smooth function on $\R^{2m}$ supported on the closed annulus 
$1\le |\tau|\le 4$ and positive on the open annulus and define
for $r>0$ 
$$ \phi_r(\tau)=\rho(2^{-r}\tau)(\sum_{k\in \Z} \rho(2^{-k}\tau))^{-1}$$
and  
$$\phi_0(\tau)=(\sum_{r\le 0}\rho(2^{-r}\tau))(\sum_{k\in \Z} \rho(2^{-k}\tau))^{-1}.$$
Then we have for all $\tau\in \R^{2m}\setminus \{0\}$,
$$\sum_{r\ge 0}\phi_r(\tau)=1.$$

We now introduce a linear map $L$ 
that will be needed in the subsequent computations.
With the  definition of $W$ and Lemma \ref{l:firstlemma} we have
$$\Theta W=\Theta D\Pi^T\tilde{\Theta} \R^m=\Theta D\Pi^T\R^m.$$
Recall from \eqref{thetanonzero}
that $\tilde{\Theta}\Pi_j \Pi^T$ is injective on $\Tilde{\Theta}\R^m$
for each $j$.
Hence $\Theta \Pi^T$ is injective
on $\Tilde{\Theta}\R^m$,
and so is $\Theta D\Pi^T$.
Therefore, the dimension of 
$$\Theta W=\Theta D\Pi^T\tilde{\Theta} \R^m$$ 
is $m-l$, the same as 
$W$.
Hence there is a  regular linear map $L: \Theta W\to W$ such that for all $\xi\in \R^m$
\begin{equation*}
D\Pi^T \tilde{\Theta} \xi=L\Theta D \Pi^T \xi.
\end{equation*}
We note that $\Theta L$ is the identity on $\Theta W$, because  for all  $\xi\in \R^m$,
with Lemma \ref{l:firstlemma},
\begin{equation}
    \label{e:olo}
    \Theta L \Theta D\Pi^T \xi= \Theta D\Pi^T \tilde{\Theta}\xi=\Theta D\Pi^T \xi .
\end{equation}
Observe for $\xi\in \R^m$
$$|L\Theta D\Pi^T\xi|= 
|D\Pi^T\tilde{\Theta}\xi|
\le d|\Pi^T\tilde{\Theta}\xi|\le Cd|\Theta \Pi^T\tilde{\Theta}\xi|=C|\Theta D\Pi^T\xi|,$$
where we used  Lemma \ref{l:firstlemma}, Lemma \ref{l:seclemma}, and  \eqref{dassumption}. 
Hence $\|L\|\le C$ with a  constant depending on $\Delta$ but not on $D$.

Using $L$, we write for the kernel in 
\eqref{e:secondkernel}
\begin{equation*}
  K(t,x) =\int_{\R^m}  |tL\Theta D \Pi^T\xi|^2  g(tD\Pi^T\xi){e^{2\pi i \xi^T  \Pi x}}  \, d\xi 
\end{equation*}
\begin{equation*}
   =C\int_V  |t L\Theta D \eta|^2 g(tD\eta){e^{2\pi i \eta^T  x}}  \, d\eta, 
\end{equation*}
with normalized Lebesgue measure on the space $V=\Pi^T\R^m$ and a constant $C$ that is controlled by $\Delta$.

We use the above partitions of unity to decompose 
\begin{equation}\label{decrsg}
\Lambda(\T, \Pi,K)=\sum_{r
\ge 0}\sum_{s\in \Z}\sum_{\gamma \in \Gamma}
\Lambda(\T, \Pi, K_{r,s,\gamma})
\end{equation}
with 
\begin{equation}
\label{e:krsg}
K_{r,s,\gamma}(t,x):=
\int_{V}  |tL\Theta D\eta |^2   g(tD\eta) \phi_r(tD\eta) \psi_\gamma (2^{s-r}tL\Theta D\eta)
e^{2\pi i \eta^T  x}\, d\eta.
\end{equation}
As $L$ is bounded, there is an $s_0\in \Z$ depending on $\Delta$ so that for 
$s<s_0$ we have
$$|2^{s-r}tL\Theta D\eta|
\le  4^{-1}|2^{-r}tD\eta|.$$
If $\phi_r(tD\eta)$ is non-zero, 
then the right-hand side of the last display is less than $1$. But then
$\psi_\gamma (2^{s-r}tL\Theta D\zeta)$ is zero. Hence we can restrict the sum in \eqref{decrsg}
to $s\ge s_0$.

We estimate each  $\Lambda(\T, \Pi, K_{r,s,\gamma})$
 separately and then sum the estimates by the triangle inequality.
Let $r,s,\gamma$
be arbitrary as in the index set of this sum and $s\ge s_0$.

We apply \eqref{decomph}  with $\xi=2^{s-r}tL\Theta D\eta$. Using 
$\Theta L\Theta=\Theta$ on $DV$ due to \eqref{e:olo}  and homogeneity of the
quadratic and linear factors, we obtain for
\eqref{e:krsg}  
\begin{equation}
\label{e:kgamma1}
\int_{V}  (t{D} \eta)^0_i(t{D} \eta)^1_{i} q_\gamma (2^{s-r}tL\Theta D \eta )  g(tD\eta) \phi_r(tD\eta) \psi_\gamma (2^{s-r}tL\Theta D\eta)
e^{2\pi i \eta^T  x}\, d\eta.
\end{equation}

Define $E$
to be the $2m\times 2m$ diagonal matrix 
$$E= 2^{s-r}\Theta D+ 2^{-r} (1-\Theta)D.$$
Define $Y:=E\Pi^T \R^m$. 
Define 
$c_{r,s,\gamma}(u)$ for $u\in Y$  
to be
\begin{equation}\label{e:defcrs}
(1+|u|^{8m})
\int_{Y}
g(\zeta )^{-1}g(DE^{-1}\zeta) \phi_r(DE^{-1}\zeta ) {\psi}_\gamma (L\Theta \zeta){q}_\gamma(L\Theta\zeta)e^{-2\pi i \zeta^T   u}\, d\zeta,
\end{equation}
where the integral is with respect to normalized Lebesgue measure on $Y$. 

By the Fourier inversion formula, we have for $\zeta \in Y$
$$g(\zeta )^{-1}g(DE^{-1}\zeta) \phi_r(DE^{-1}\zeta) \psi_\gamma (L\Theta \zeta)q_\gamma (L\Theta \zeta )=
\int_{Y}\frac {c_{r,s,\gamma}(u)}{1+|u|^{8m}} e^{2\pi i u^T \zeta }\, du$$
and note that $c_{r,s,\gamma}(u)$ is Schwartz in $u$.
Applying this with $\zeta=tE\eta$ gives for \eqref{e:kgamma1}
$$\int_{Y}\frac{2^{2r-2s}c_{r,s,\gamma}(u)}{1+|u|^{8m}}
\int_{V} e^{2\pi i \eta^T   x} 
(tE \eta)_i^0(tE \eta)_i^1
 g(tE \eta) 
 e^{2\pi i u^T  tE\eta}
 \, d\eta\, du .$$
We identify the inner integral as an integral over the range of $\Pi^T$ of the  Fourier transform of $F$ given by
 $$F(p)=(\partial_i^0 \partial_i^1 g)_{tE}(x+p+tEu),$$
 as then
$$\widehat{F}(\eta)=4\pi^2  
(tE \eta)_i^0(tE \eta)_{i}^1 g(tE \eta) e^{2\pi i \eta ^T  tE u } e^{2\pi i \eta^T  x}.$$

We may replace the inner integral by an integral of $F$ itself over the orthogonal complement of the range of $\Pi^T$. Thus we obtain with a constant $C$ for \eqref{e:kgamma1}
\begin{equation}
   \label{e:ftside} \int_{Y} \frac{2^{2r-2s}c_{r,s,\gamma}(u)}{1+|u|^{8m}} \int_{\R^m}    (\partial_i^0 \partial_{i}^1 g)_{tE}(x+Jp+tEu)\,dp \, du .
\end{equation}

 Using the assumed bound
\eqref{e:indbound}
for 
\eqref{e:firstkernel} and integrability of 
$(1+|u|^{4m})^{-1}$ over
$Y$, we obtain 
the bound
\begin{equation}\label{e:refquestion}|\Lambda(\T,\Pi,K_{r,s,\gamma})((F_j)_{j\in \C})|
\end{equation}
$$\le C
2^{2r-2s}\|c_{r,s,\gamma}\|_\infty \Big(\prod_{i=1}^m \max\{e_i^{2\alpha},e_i^{-2}\}\Big)  |Q_{\T}| \prod_{j\in \C} M_\alpha(F_j)(A_j\T)$$
$$\le C
2^{(2+2m)r +(2m\alpha-2)s}\|c_{r,s,\gamma}\|_\infty \Big(\prod_{i=1}^m \max\{{d}_i^{2\alpha},{d}_i^{-2}\}\Big)  |Q_{\T}| \prod_{j\in \C} M_\alpha(F_j)(A_j\T).$$
Here we used that $r\ge 0$ and $s\ge s_0$ and waived the potential gain from any powers of $2$ by negative multiples of $r$ or $s$.
 
To prove the desired bound
\eqref{e:indbound}
for 
\eqref{e:secondkernel},
it remains to show
\begin{equation}\label{e:csum1}
\sum_{r\ge 0} \sum_{s\ge s_0} \sum_{\gamma \in \Gamma}2^{(2+2m)r+(2m\alpha-2)s}\|c_{r,s,\gamma}\|_\infty\le C.
\end{equation}

 To see 
\eqref{e:csum1}, it suffices to
show for each $r,s ,\gamma$
\begin{equation*}
    \|c_{r,s,\gamma}\|_\infty\le C 2^{-{2^r}}.
\end{equation*}
Decompose $\zeta$ in 
$Y$ into $\Theta\zeta$ and $(1-\Theta)\zeta$
and note
$$DE^{-1}\zeta= 2^{r-s}\Theta \zeta+ 2^r (1-\Theta)\zeta.$$
For $\psi_\gamma(L \Theta\zeta)$ to be non-zero,
we need $\Theta\zeta$
to be bounded by some constant depending on $\Delta$.
For $\phi_r(DE^{-1}\zeta)$ to be non-zero,
we need 
$(1-\Theta)\zeta$ to be bounded by $C$.
Hence the integrand 
in \eqref{e:defcrs}
is supported in a ball of radius $C$.
It will then suffice to show that the partial derivatives
of the function 
$$\zeta \mapsto g(\zeta )^{-1}g(DE^{-1}\zeta) \phi_r(DE^{-1}\zeta ) {\psi}_\gamma (L\Theta \zeta){q}_\gamma(L\Theta\zeta)$$
satisfy a bound
$C2^{-2^r}$
with constant depending on $\Delta$ and the order of derivative.
The smallness ultimately comes from the smallness of the Gaussian $g(DE^{-1}\zeta)$
on the support
of the function $\phi_r$.
More precisely, as all factors are smooth, it suffices to
assume $r>2$.
The first factor and last two factors do not depend on $r$
and are compactly supported and smooth, it therefore suffices to show
the analogous function value and derivative bound on
$$\zeta \mapsto g(DE^{-1}\zeta) \phi_r(DE^{-1}\zeta ).$$
However, 
$$\phi_r(DE^{-1}\zeta )$$ is nonzero only on
an annulus of radius $2^r$, where the Gaussian is very small.
The desired bound then follows from elementary calculus.
This completes the reduction of
bounds for \eqref{e:secondkernel} to
bounds for \eqref{e:firstkernel}.

We conclude this section with remarks on the cases  $J=(I,I)^T$, i.e.,  $\Pi=(I,-I)$, or $m=2$,
in which the above proof can be simplified.
If $\Pi=(I,-I)$ then we have
$$
|D\Pi^T\tilde{\Theta}\xi|^2 = 2\sum_{l<i\le m} d_i^2 \xi_i^2,
$$
while 
$$
\sum_{l<i\le m} (D\Pi^T\xi)_i^0 (D \Pi^T \xi)_{i}^1=-\sum_{i=l+1}^m d_i^2 \xi_i^2.
$$
Using the expression for the  kernel~\eqref{e:firstkernel} on the frequency side as between \eqref{e:kgamma1} and 
   \eqref{e:ftside}, 
the kernel \eqref{e:secondkernel} is trivially a linear combination of kernels
\eqref{e:firstkernel} with $u=0$ and $c(t)=1$.
Similarly, if $l=m-1$, which is the only case that needs to be discussed if $m=2$, then 
both
$
|D\Pi^T\tilde{\Theta}\xi|^2
$ 
and
$
(D\Pi^T\xi)_m^0 (D \Pi^T \xi)_{m}^1
$
are non-zero multiples of 
$d_m^2 \xi_m^2$, 
which simplifies the proof.

However, in general,
$$
|D\Pi^T\tilde{\Theta}\xi|^2
$$ 
 is  not a  linear combination  of terms of the form
 $
(D\Pi^T\xi)_i^0 (D \Pi^T \xi)_{i}^1
$, where $l<i\le m$. To see this, take $m=3$, $l=1$, $D=I$ and
$\Pi=(I,A)$ with 
\[A = \left ( \begin{array}{ccc}
-1 & 0 & 0 \\
a & -1 & b \\
c & d & -1
\end{array} \right ).\]
Then 
\begin{equation*}
|D\Pi^T\tilde{\Theta}\xi|^2  =   |(I,A)^T(0,\xi_2,\xi_3) |^2  =\xi_2^2 +\xi_3^2+(a\xi_2+c\xi_3)^2 +(-\xi_2+d\xi_3)^2+(b\xi_2-\xi_3)^2
\end{equation*}
\begin{equation*}
=(2+a^2+b^2)\xi_2^2 +(2+c^2+d^2)\xi_3^2 +2(ac-b-d)\xi_2 \xi_3.
\end{equation*}
Further,
$$
(\Pi^T\xi)^0_2(\Pi^T\xi)^1_2 =  \xi_2 (A^T \xi)_2=-\xi_2^2+d\xi_2 \xi_3
$$
and
$$
(\Pi^T\xi)^0_3(\Pi^T\xi)^1_3  =  \xi_3 (A^T \xi)_3=-\xi_3^2+b\xi_2 \xi_3.
$$
Thus, comparing the coefficients at $\xi_2^2$ and $\xi_3^2$, we see that in order to have
\begin{equation}\label{E:lin_comb}
|D\Pi^T\tilde{\Theta}\xi|^2  = \alpha (\Pi^T\xi)^0_2(\Pi^T\xi)^1_2  +\beta (\Pi^T\xi)^0_3(\Pi^T\xi)^1_3
\end{equation}
one necessarily needs
$\alpha=-(2+a^2+b^2)$ and $\beta=-(2+c^2+d^2)$. With this choice of $\alpha$, $\beta$, equation~\eqref{E:lin_comb} is fulfilled if and only if the corresponding coeffients at $\xi_2 \xi_3$ on both sides of~\eqref{E:lin_comb} coincide, which amounts to
$$
a^2 d+ b^2 d+bc^2+bd^2+2ac=0.
$$
The last equality is   not satisfied in general.

\subsection{Telescoping: Bounds for the second kernel imply bounds for the third kernel}

\label{S:telescoping}
Let  $\Delta$, $J$, $A_j, \alpha, l,
D, F_j$ be as in Proposition \ref{p:induction}.
Assume $0<l$.

 Let $K$ be the kernel  \eqref{e:thirdkernel}, it is 
a translation invariant  integral of the function 
$$y\mapsto  \sum_{1\le i\le l}(\partial^0_i \partial^1_{i} g)_{t{D}}(x+y)$$
over the kernel of $\Pi$. This is equal to a translation invariant integral of the Fourier transform of the function over the range  $\Pi^T$.
Therefore, with a constant $C$ controlled by $\Delta$,
\begin{equation}\label{e:thirdfourier}
 K(t,x)=  C \sum_{1\le i\le l} \int_{\R^m}    (tD\Pi^T\xi)^0_i
(tD\Pi^T\xi)^1_{i} g(tD\Pi^T\xi )  e^{2\pi i x^T  \Pi^T \xi}\,d\xi
\end{equation}
\begin{equation*} 
   =C \int_{\R^m} \langle (1-\tilde{\Theta})\Pi_0(tD\Pi^T\xi), (1-\tilde{\Theta})\Pi_1(tD\Pi^T\xi)\rangle g(tD\Pi^T\xi )  e^{2\pi i x^T \Pi^T \xi}\,d\xi.  
\end{equation*}
Note that for $\eta\in \R^{2m}$ 
we have
\begin{equation}\label{E:firstdisplay}
|\eta^2|+2\left<(1-\tilde{\Theta})\Pi_0\eta, (1-\tilde{\Theta})\Pi_1\eta\right>
\end{equation}
\begin{equation}\label{e:lastdisplay}=|\tilde{\Theta} \Pi_0\eta|^2+|\tilde{\Theta}\Pi_1\eta|^2+ |(1-\tilde{\Theta}) (\Pi_0+\Pi_1)\eta|^2.\end{equation}
The expression
\eqref{e:lastdisplay}
takes the same value for the vectors
$\eta=D\Pi^T \xi$ 
and 
$\eta=D\Pi^T \tilde{\Theta}\xi$
with $\xi\in \R^m$.
This follows individually for the first two terms in
\eqref{e:lastdisplay}
from
Lemma
\ref{l:firstlemma} and for the third term from
\eqref{e:ej}
in the form
$$
(\Pi_0+\Pi_1)\Pi^T(1-\tilde{\Theta})=0.$$
Hence the expression 
\eqref{E:firstdisplay}
also takes the same value for these vectors and we obtain
$$|D\Pi^T \xi^2|+2\left<(1-\tilde{\Theta})\Pi_0D\Pi^T \xi, (1-\tilde{\Theta})\Pi_1D\Pi^T \xi\right>=
|D\Pi^T\tilde{\Theta}\xi|^2,$$
where we have dropped the inner product term on the right hand side because it vanishes because $\Pi_0\Pi^T$ is the identity and $(1-\tilde{\Theta})\Pi_0\Pi^T\tilde{\Theta}=0$. We remark that the vanishing of this last term is due to our choice of normalization of $\Pi=\Pi^J$ with $\Pi^J \circ J=0$. If one worked with other choices, one would use the full
quadratic form 
\eqref{E:firstdisplay} for $\eta=D\Pi^T\tilde{\Theta}\xi$ in the 
definition of the kernel \eqref{e:secondkernel}.

Then the kernel  $2K(t,x)$ is a difference of 
\[
\tilde{\tilde{K}}(t,x):= C \int_{\R^m} |tD\Pi^T\tilde{\Theta}\xi|^2 g(tD\Pi^T\xi )  e^{2\pi i x^T \Pi^T \xi}\,d\xi
\]
and
\[ \tilde{K}(t,x):= C \int_{\R^m} |tD\Pi^T\xi|^2  g(tD\Pi^T\xi )  e^{2\pi i x^T \Pi^T \xi}\,d\xi.
\]
By the induction hypothesis, 
$\Lambda(\mathcal{T},J,\tilde{\tilde{K}})$
satisfies  the analogue of \eqref{e:indbound}. 

Thus, to show the desired bound \eqref{e:indbound}   for $\Lambda(\mathcal{T},J,K)$, 
it suffices to show the analogue of \eqref{e:indbound} for  
$\Lambda(\mathcal{T},J,{\tilde{K}}).$
Analogous Fourier transform arguments as the ones leading to \eqref{e:thirdfourier}, put in reverse, give
\[ \tilde{K}(t,x) = \int_{\R^m} (\Delta g)_{tD}  (x+Jp)\, dp.
\]
As is well known, the gaussians give fundamental solutions to the heat equation, more precisely, we have
\[(\Delta g)_{tD}(x)   =
t^{-2m}(\det D)^{-1}(-4\pi m +4\pi^2 |(tD)^{-1}x|^2)e^{-\pi |(tD)^{-1}x|^2}
=  2\pi t \partial_t(g_{tD}).\]

Let $k_\T$ be defined by $l(Q_\T)=2^{k_\mathcal{T}}$. 
Using the heat equation in the kernel of $\tilde{K}$  we obtain
$$(2\pi)^{-1} \Lambda(\T,J,\tilde{K}) =
 \sum_{k\in \Z}  \int_{\R^m} \int_{\R^{2m}}  \Big(\prod_{j\in \C} (F_j 1_{A_jT_k}(\Pi_jx) \Big) \int_{2^{k-1}}^{2^k}
t\partial_{t}(g_{tD})(x+Jp) \, \frac{dt}{t} 
 \,dx\, dp.$$
 By the fundamental theorem of calculus in $t$, this equals
$$ \sum_{k\in \Z} \int_{\R^m}\int_{\R^{2m}} \Big( \prod_{j\in \C} ({F}_j1_{A_jT_k}) (\Pi_j x)  \Big ) (g_{2^{k}D} - g_{2^{k-1}D})(x+Jp) \, dx \, dp  $$
\begin{equation}\label{e:teles}
= \int_{\R^m}\int_{\R^{2m}} \Big( \prod_{j\in \C} ({F}_j1_{A_jQ_\T}) (\Pi_j x)  \Big )   g_{2^{k_\T}D}(x+Jp)\, dx \, dp   
\end{equation}
\begin{equation}\label{e:teles1}
  + \sum_{k<k_\T} 
\int_{\R^m}\int_{\R^{2m}} \Big( \prod_{j\in \C} ({F}_j1_{A_jT_k}) (\Pi_j x) 
- \prod_{j\in \C} ({F}_j1_{A_jT_{k+1}}) (\Pi_j x)  \Big )
 g_{2^kD}(x+Jp) \, dx \, dp .  
 \end{equation}
We estimate the two terms \eqref{e:teles} and \eqref{e:teles1} separately.

We begin with \eqref{e:teles1}. We estimate 
\begin{equation}\label{e:stepfunction}
g\le C \sum_{n\ge 0} e^{-2^{n}}\chi_{2^n},
\end{equation}
where $\chi$ is the characteristic function of the square $[-1,1]^{2m}$ and $C$ is a constant  that depends only on $m$. We insert \eqref{e:stepfunction} into 
\eqref{e:teles1}. 
We fix $k<k_\T$ and $n\ge 0$ and consider
\begin{equation*}
 \int_{\R^m}\int_{\R^{2m}} \Big( \prod_{j\in \C} ({F}_j1_{A_jT_k}) (\Pi_j x) 
- \prod_{j\in \C} ({F}_j1_{A_jT_{k+1}}) (\Pi_j x)  \Big )
 \chi_{2^{n+k}D}(x+Jp) \, dx \, dp   .  
\end{equation*} 
Using the distributive law and $T_k\subseteq T_{k+1}$, we have 
$$\Big| \prod_{j\in \C} ({F}_j1_{A_jT_k}) (\Pi_j x) 
- \prod_{j\in \C} ({F}_j1_{A_jT_{k+1}}) (\Pi_j x)  \Big |$$
$$\le \sum_{j_0\in \C}
|{F}_{j_0}|1_{A_{j_0}T_{k+1}\setminus A_{j_0}T_{k}} (\Pi_{j_0} x) 
\prod_{j\neq j_0}
|{F}_j|1_{A_jT_{k+1}} (\Pi_j x). $$
Fix $j_0$. For simplicity of notation consider $j_0=0$, other values of $j_0$ will be analogous.

Let $Q$ be a cube of side length $2^k$ contained in $T_{k+1}\setminus T_k$. We consider
\begin{equation}
    \label{e:teles3}
 \int_{\R^m}\int_{\R^{2m}} 
 |{F}_{0}|1_{A_{0}Q} (\Pi_{0} x) 
\Big( \prod_{j\neq 0}
|{F}_j|1_{A_jT_{k+1}} (\Pi_j x) \Big )
  \chi_{2^{n+k}D}(x+Jp)\, dx \, dp   .  
\end{equation} 
Let $E$ be the $2m \times 2m$ diagonal
matrix such that for all $1\leq i \leq m$ and $\sigma=0,1$,
$${e}^\sigma_{i}=\max (d^\sigma_{i},1).$$
Denote by $\tilde{E}$ the upper left $m\times m$ block of $E$.
If the integrand in \eqref{e:teles3} is not zero, then
\begin{equation}\label{e:set1}
 \Pi_0x\in A_0Q,   
\end{equation} 
\begin{equation}\label{e:set2}
 x+Jp\in 
 {2^{n+k}D}[-1,1]^{2m}
\subseteq {2^{n+k}E}[-1,1]^{2m}.
\end{equation}
Let us denote \[S:={2^{n+k}\tilde{E}}[-1,1]^{m}.\]
Applying $\Pi_0$ and $\Pi_1$ to 
\eqref{e:set2}, we obtain
\begin{equation}\label{e:set3}
 \Pi_0x+A_0p\in  S,
\end{equation}
\begin{equation}\label{e:set4}
 \Pi_1x+A_1p\in  S.
\end{equation}
Combining \eqref{e:set1} and \eqref{e:set3}, we obtain
\begin{equation}\label{e:set5}
A_0p\in S -A_0Q.
\end{equation}
Applying $A_1A_0^{-1}$, we obtain
\begin{equation}\label{e:set6}
A_1p\in A_1A_0^{-1}S- A_1Q.
\end{equation}
Combining \eqref{e:set4} with \eqref{e:set6}, we obtain
\begin{equation}\label{e:set7}
\Pi_1x\in 
S-A_1A_0^{-1}S+ A_1Q.
\end{equation}

We cover the set in \eqref{e:set7} by pairwise disjoint boxes
$$ P=P_1\times \dots \times P_m,$$ 
for example with corners in a suitable lattice, such that
$|P_h|= 2^n2^k\tilde{e}_h$.
Let $\mathcal{P}$ denote the set of boxes
in this cover.
We estimate \eqref{e:teles3} using the cover $\mathcal{P}$
and then integrating $p$
by
\begin{equation}
\label{e:teles4}
\sum_{P\in \mathcal{P}} \frac{|S_P|}{2^{2m(k+n)}\det D} \int_{\R^{2m}}  |{F}_{0}|1_{A_{0}Q} (\Pi_{0} x) 
 |{F}_{1}|1_{A_{1}T_{k+1}\cap P} (\Pi_{1} x) 
\Big( \prod_{j\neq 0,1}
|{F}_j|1_{A_jT_{k+1}} (\Pi_j x) \Big ) \,
 dx
\end{equation}
where 
$$S_P:=(A_1^{-1}(S-P))\cap (A_0^{-1}S-Q)$$
contains the essential support of $p$ by \eqref{e:set4}, \eqref{e:set5}, and  $\Pi_1x\in P$.

We cover $A_0Q$ by a rectangular box
$$A_0Q\subseteq I_1^0\times \dots \times I_m^0$$
such that for each $1\le h\le m$ we have 
$|I^0_h|\le C2^k$ with $C$ depending on $\Delta$.
We fix one box $P\in {\mathcal P}$ and write
$I_h^1$ for $P_h$.
Write
$$I_j=I_1^{j(1)}\times \dots \times I_m^{j(m)}.$$
Then we estimate the summand in \eqref{e:teles4} for this box $P$ by
\begin{equation*}
 \frac{|S_P|}{2^{2m(k+n)}\det D}
 \int_{\R^{2m}} 
 \Big( \prod_{j\in \C}
|{F}_j|1_{A_jT_{k+1}\cap I_j} (\Pi_j x) \Big ) \,
 dx.
\end{equation*}
Here we used that $Q\subseteq T_{k+1}$ and that  $\Pi_0x\in I_0$ together with
$\Pi_1x\in I_1$ implies
$\Pi_jx\in I_j$.

Applying the Brascamp-Lieb inequality in Lemma  \ref{l:brascamplieb} estimates the last display by
\begin{equation*}
\leq C 
 \frac{|S_P|}{2^{2m(k+n)}\det D}
  \prod_{j\in \C}
|I_j|^{2^{1-m}}[{F}_j1_{A_jT_{k+1}}]_{I_j}   
\end{equation*}
\begin{equation*}
\label{e:teles7}
\le  C 
\frac{|S_P|\det \tilde{E}}{2^{mn}\det D}(1+|\tilde{E}|)^{2\alpha}
  \prod_{j\in \C}
M_\alpha ({F}_j)(A_j\T)   .
\end{equation*}
Here we estimated  averages by the modified maximal functions using that {if $A_j T_{k+1}$ has a nonempty intersection with $I_j$ then} the box
$CI_j$ for suitable $C$ depending on $\Delta$ covers the set {$A_jQ'$} for some $Q'\in \T$. To estimate the product of $|I_j|^{2^{1-m}}$ we used that each $I_h^\sigma$ appears in the product $2^{m-1}$ times.

The {parallelepipeds} $A_1^{-1}(P-S)$ have  bounded overlap depending on $m$   
as $P$ runs through $\mathcal{P}$, because  the box $P-S$ is contained in {$3P$}, which intersects {$3P'$} only
for near neighbors $P'$ of $P$. 
Hence
$$\sum_{P\in \mathcal{P}} |S_P|\le C|A_0^{-1}S-Q|\le 
C2^{m(n+k)} \det \tilde{E}.$$ 
Here, the last inequality is obtained as follows. 
The set  $S$
 can be covered by
$C2^{mn}  \det \tilde{E}$
dyadic boxes of side length $2^k$.
Then the set $A_0^{-1}S$ can also
be covered by
$C2^{mn}  \det\tilde{E}$
dyadic boxes of side length $2^k$
with a constant depending on $\Delta$. 
Finally, the set $A_0^{-1}S-Q$
can be covered 
by
$C2^{mn}  \det \tilde{E}$
dyadic boxes of side length $2^{k+2}$. This gives the above 
volume estimate.

Summing over the boxes in $\mathcal{P}$ 
estimates \eqref{e:teles3} by
\begin{equation*}
 C \frac{\det \tilde{E} ^2}{\det D }
(1+|\tilde{E}|)^{2\alpha}
2^{km}  \prod_{j\in \C}
M_\alpha ({F}_j)(A_j\T)   .
\end{equation*}
Summing over the various $j_0\in \C$, all estimated analogously to $j_0=0$,
and then summing over the disjoint cubes $Q$ of volume $2^{km}$ in $T_{k+1}\setminus T_k$
and 
summing over $k<k_\T$,
using that the $T_{k+1}\setminus T_k$ are disjoint in $Q_\T$, and summing over $n$ estimates
\eqref{e:teles1} by
\begin{equation*}
 C \frac{\det \tilde{E}^2}{\det D}
(1+|\tilde{E}|)^{2\alpha}
 |Q_\T|  \prod_{j\in \C}
M_\alpha ({F}_j)(A_j\T) 
\end{equation*}
\begin{equation*}
=  C 
\Big(\max_{i=1,\dots,m} (1+d_i)^{2\alpha} \prod_{i=1}^m \max\{d_i^{-2},1\} \Big) |Q_\mathcal{T}|   \prod_{j\in \C}
M_\alpha ({F}_j)(A_j\T)   .
\end{equation*}

It remains to estimate \eqref{e:teles}, which is done similarly as \eqref{e:teles1} but simpler. 
Estimating the gaussian by a superposition of characteristic functions of boxes we consider
\begin{equation*}
 \int_{\R^m}\int_{\R^{2m}} 
\Big( \prod_{j\in \C}
|{F}_j|1_{A_jQ_{\T}} (\Pi_j x) \Big )
  \chi_{2^{n+k}D}(x+Jp) \, dx \, dp ,
\end{equation*} 
which then is estimated analoguously to   \eqref{e:teles3}.

\subsection{Cauchy-Schwarz: Bounds for the third kernel imply bounds  for the first kernel one level down}

Let  $\Delta$, $A_j, \alpha, l,
D, F_j$ be as in Proposition \ref{p:induction}.
Let $c,u,i$ be  as in \eqref{form:decomposed}
and assume $i>l$.
Without loss of generality we assume $i=l+1$,
as we may interchange the indices $i$ and $l+1$
if necessary in the desired estimate and all assumptions remain intact.

Consider the quantity
$$(1+|u|^{4m})\Lambda(\T,J,K)((F_j)_{j\in \C})$$
for the kernel $K$ defined by \eqref{e:firstkernel}. This quantity is equal to
\begin{equation}\label{beforecs}
   \int_{\R^{2m}}\int_0^\infty  \int_{\R^m} c(t) \Big(\prod_{j\in \C} (F_j 1_{A_jT_t})(\Pi_jx) \Big) 
   (\partial_{i}^0 \partial_{i}^1 g)_{tD}(x+Jp+tDu)\,dp \,
\frac{dt}{t} \,dx,
\end{equation}
where we have set $T_t:=T_k$ if $2^{k-1}\le t<2^{k}$.
We change the order of integration, so that the innermost
integration variables are $x_i^0$ and $x_i^1$. These
two integrals separate. Hence we obtain for 
\eqref{beforecs}
$$\int_{\R^{2m-2}}\int_0^\infty  \int_{\R^m} c(t)\prod_{\sigma=0}^1
\Big (
\int_\R\Big(\prod_{ j(i)=\sigma} (F_j 1_{A_jT_t})(\Pi_jx) \Big) 
({g}')_{td_i}((x+J p+tDu)_i^\sigma) dx_i^\sigma \Big )\, d\mu$$
with
 \begin{equation*}
 d\mu=
 \prod_{1\le h\le m, h\neq i} \Big (  \prod_{\sigma =0}^1  g_{td_h}
((x+J p+tDu)_h^\sigma) dx_h^\sigma \Big)\,dp \,
\frac{dt}{t},
\end{equation*}
where $g$ denotes the one-dimensional Gaussian.

We estimate $|c(t)|$ by $1$ and apply Cauchy-Schwarz in the Hilbert space $L^2(d\mu)$  to estimate the modulus square of \eqref{beforecs} by
\begin{equation}\label{aftercs0}
  \prod_{\sigma=0}^1 \int_{\R^{2m-2}}\int_0^\infty  \int_{\R^m}
\Big(
\int_\R\Big(\prod_{ j(i)=\sigma} (F_j 1_{A_jT_t})(\Pi_jx) \Big) 
({g}')_{td_i}((x+J p+tDu)_i^\sigma) dx_i^\sigma \Big )^2 \, d\mu  .
\end{equation}

Fix $\sigma \in \{0,1\}$. 
The matrix $J^\sigma$ is invertible.
Define
$q\in \R^m$ by  $q :=  - (J^\sigma)^{-1}v$, where $v_i=u^\sigma_i$ and $v_h=0$ for $h\neq i$.
Then we have  $Jq=\Theta Jq$, as one can see from
$$(J^0 q)_h=(J^1 q)_h= -v_h=0$$ 
for $1\le h\le i-1$ as the first $i-1$ rows of $J^0$ and $J^1$ coincide.
As a consequence, using that $d_h=d_i$ for $h\ge i$,   we have
$$D J  q={J}  {q}d_i .$$
Moreover, 
$$|q|\le C|u|$$
with a constant $C$ depending on $\Delta$.
For fixed $t$, we 
use an affine change of variables  to substitute $p$ by $$p+t q d_i$$
and we then obtain for the $\sigma$-th factor in
\eqref{aftercs0} 
\begin{equation}\label{aftercs1}
  \int_{\R^{2m-2}}\int_0^\infty  \int_{\R^m}
\Big(
\int_\R\Big(\prod_{ j(i)=\sigma} (F_j 1_{A_jT_t})(\Pi_jx) \Big) 
({g}')_{td_i}((x+J p)_i^\sigma) dx_i^\sigma \Big )^2 \, d\tilde{\mu}  ,
\end{equation}
where 
\begin{equation*}
d\tilde{\mu}:=\prod_{1\le h\le m, h\neq i} 
\prod_{\lambda=0}^1 g_{td_h}
((x+ Jp  +tD(Jq+u))^\lambda_h)\,dp \frac{dt}{t}.
\end{equation*}
The square in    \eqref{aftercs1} is  positive. We
therefore may estimate  \eqref{aftercs1} by replacing the measure $d\tilde{\mu}$ by the larger measure given by
\begin{equation}\label{domgaussweight}
 C(1+|u|)^{2m-2}\prod_{1\le h\le m, h\neq i}
\prod_{\lambda=0}^1 g_{2td_h(1+|u|)}
((x+Jp )_h^\lambda )\,dp \frac{dt}{t}
\end{equation}
with a constant $C$ depending  on $\Delta$. 
Here we estimated for each $h$
$$(td_h)^{-1}
g
((td_h)^{-1}(x+Jp)_h^\lambda +(Jq+u)^\lambda_h)\le 
C(td_h)^{-1}g
((2+2|u|)^{-1} (td_h)^{-1}(x+Jp)^\lambda_h)
$$
using the elementary inequality
$$g(a+b)\le C g((2+2d)^{-1}a)$$
for real $a,b$  and positive $c,d$
with $|b|\le c{d}$ and
a constant $C$ depending only on $c$.
The elementary inequality is proven by considering separately the case
$|a|\le 2|b|$,
in which the right-hand side is larger than $1$ for sufficiently large $C$, and
$|a|> 2|b|$, in which the inequality holds with constant $C=1$ by monotonicity of $g$ on either half axis.

Performing the above for each $\sigma=0,1$ in \eqref{aftercs0} we estimate 
$$(1+|u|)^{4m}\Lambda(\T,J,K)((F_j)_{j\in \C})^2$$
\begin{equation*}
\le C\prod_{\sigma=0}^1  \int_{\R^{2m-2}}\int_0^\infty  \int_{\R^m}
 \Big(
\int_\R\Big(\prod_{ j(i)=\sigma} (F_j 1_{A_jT_t})(\Pi_jx) \Big) 
({g}')_{td_i}((x+J p)_i^\sigma) dx_i^\sigma \Big )^2 
\end{equation*}
\begin{equation}\label{aftercsest}
\times \prod_{1\le h\le m, h\neq i}
\prod_{\lambda=0}^1 g_{2td_h(1+|u|)}
((x+Jp )_h^\lambda )\,dp \, \frac{dt}{t}.
\end{equation}

We aim to interpret the two
factors with $\sigma=0,1$ as forms $\Lambda$ as in
\eqref{deflambda4} for some modification of $J$. 
For $\sigma=0,1$ we define the matrix $_\sigma J$ as follows. For 
each $1\le h\le m$ and each $\lambda=0,1$ we set
$$ _{\sigma }J^{\lambda}_h:=J^\lambda_h$$
if $\lambda =\sigma $ or $h\neq i$ and 
$$ _{\sigma}J^{\lambda}_h:=J^{1-\lambda}_h$$
if $\lambda \neq \sigma$ and $h=i$. 
 For $\sigma=0,1$ we define $_\sigma F_j:=F_j$ 
if $j(i)=\sigma $ and
$_\sigma F_j:=F_{i*j}$ if 
$j(i)\neq \sigma$. 
We also  define $ _\sigma A_j:=\Pi_j (_\sigma J)$ so that if
$j(i)=\sigma$ we have $A_j={_\sigma A_j}$ and if $j(i)\neq \sigma$
then we have $A_{i* j}={_\sigma A_j}$.
In particular, we see that the matrices
${_\sigma A_j}$ are $\Delta$-regular.

Let    $E$  be a $2m\times 2m$ diagonal matrix with two identical $m\times m$ diagonal blocks, where    the first $m$ diagonal elements  $e_1,\ldots, e_m$ of $E$ are defined by   $e_h =2(1+|u|)d_h$ if $h\neq i$ and $e_i =d_i$. 
The matrix $E$ satisfies the analogue of \eqref{dassumption} with $l$ replaced by $i$. 
Then we expand out the square in  \eqref{aftercsest} and regroup the terms to write \eqref{aftercsest} as
\begin{equation}\label{aftercs}
\prod_{\sigma=0}^1   \int_{\R^{2m}}\int_0^\infty  \int_{\R^m}  \Big(\prod_{j\in \C} (_\sigma F_j 1_{_\sigma A_jT_t})(\Pi_jx) \Big) 
   (\partial_{l+1}^0 \partial_{l+1}^1 g)_{tE}(x+ {_\sigma J} p)\,dp \,
\frac{dt}{t} \,dx
\end{equation}
\[=\prod_{\sigma=0}^1\Lambda(\mathcal{\T},{_\sigma}J,{_\sigma}K_{l+1})((_\sigma F_j)_{j\in \C}) ,  \]
 where we used $i=l+1$ and  for $1\le \kappa \le l+1$ we defined
 \[_\sigma K_{\kappa} (t,x) := \int_{\R^m} (\partial_{\kappa}^0 \partial_{\kappa}^1 g)_{tE}(x+ {_\sigma J} p)\,dp .\]
We observe  that for $1\leq \kappa \le l+1$
\begin{equation}\label{e:positivity}
\Lambda(\mathcal{\T},{_\sigma}J,{_\sigma}K_{\kappa})((_\sigma F_j)_{j\in \C}) \geq 0 .
\end{equation}
For $\kappa=l+1$ this follows from evident positivity of \eqref{aftercsest}. Similarly we can use the symmetries of ${_\sigma}F_j$ and $_\sigma J$ to write 
$$\Lambda(\mathcal{\T},{_\sigma}J,{_\sigma}K_{\kappa})((_\sigma F_j)_{j\in \C})$$
for $\kappa<l+1$ in analogous form to \eqref{aftercsest}.
Thus, to estimate   \eqref{aftercs}   it suffices to estimate \[\prod_{\sigma=0}^1 \Lambda(\mathcal{\T},{_\sigma}J,\sum_{\kappa=1}^{l+1} {_\sigma}K_{\kappa})((_\sigma F_j)_{j\in \C}).  \]
We identify two kernels as in 
\eqref{e:thirdkernel}. Using \eqref{e:indbound} for $l+1$
we estimate the last display by
$$ 
\leq C\prod_{\sigma =0}^1  \,
\Big(\prod_{i=1}^m \max\{e_i^{2\alpha},e_i^{-2}\}\Big)  |Q_{\T}| \prod_{j\in \C} M_\alpha(_\sigma F_j)(_\sigma A_j\T)
$$
\begin{equation}\label{lastsillythingnov22}=
C
\Big(\prod_{i=1}^m \max\{e_i^{2\alpha},e_i^{-2}\}\Big)^2 \Big( |Q_{\T}| \prod_{j\in \C} M_\alpha(F_j)(A_j\T)\Big)^2 .
\end{equation}
Here we reshuffled the products using
$$M_\alpha({_\lambda F_j})(_\lambda A_j\T)=
M_\alpha(F_j)(A_j\T)$$ 
if
$j(i)=\lambda$ and 
$$M_\alpha({_\lambda F_j})(_\lambda A_j\T)=M_\alpha(F_{i*j}) (A_{i*j}\T)$$
if $j(i)\neq \lambda$. We use the definition of $e_i$  to  estimate   \eqref{lastsillythingnov22} as
$$\le (1+|u|)^{4m}
C
\Big(\prod_{i=1}^m \max\{{d}_i^{2\alpha},{d}_i^{-2}\}\Big)^2 \Big( |Q_{\T}| \prod_{j\in \C} M_\alpha(F_j)(A_j\T)\Big)^2 .
$$
This proves the desired estimate.

\section{Proof of Theorem \ref{T:sharp}}

\label{S:sharpness}
 Assume to get a contradiction that
there is $q<2^{m-1}$ so that the analog of Theorem \ref{mainp} holds with the lower threshold $2^{m-1}$ in \eqref{E:prange} replaced by $q$.

Set $\Pi=(I,I)$ so that $\Pi_j\Pi^T=I$  for all $j\in \C$.
Let $K$ be the Riesz kernel 
$$
K(x)=\frac{x_1}{|x|^{m+1}}, \quad x\in \R^m \setminus\{0\}.
$$
For $j\in \C$ with $j(1)=0$ let
 \[p_j=2^{m-1}\]
 and for $j\in \C$ with  $j(1)=1$ let \[p_j=\infty.\] 
These values $p_j$ satisfy the assumptions of the analog of Theorem \ref{mainp}. 

Let $N$ be a large positive number.
For $j\in \C$ with  $j(1)=0$,  let  $F_j:\R^m\to [0,1]$
be smooth, supported in the box
$$[0,1]\times [-2N,2N]^{m-1}$$
and taking value $1$ on the smaller box $$[1/4,1/2]\times [-N,N]^{m-1}.$$
Then $$\|F_j\|_{2^{m-1}} \leq   (4N)^{(m-1)2^{1-m}}.$$ 

Let 
 $F_1:\R^m\to [0,1]$ be smooth,
 supported in the box
 $$[1/2,2N]\times [-2N,2N]^{m-1}$$ and taking value $1$ on the smaller box $$[1,N]\times [-N,N]^{m-1}. $$
 Finally, for $j\in \C$ with  $j(1)=1$ but $j\neq 1$, let $F_j:\R^m\to [0,1]$ be smooth and compactly supported, taking value $1$ on the box $[-N,N]^m$.

With $ \|F_j\|_{\infty}=1$ for all $j\in \C$, we obtain by assumption a constant $C$ independent of $N$ so that 
\begin{equation}
    \label{lastsillythingnov25}
\Big|p.v. \int_{\R^{2m}}  K(\Pi x) \prod_{j\in \C} F_j(\Pi_jx )\, dx \Big|\le C\prod_{j\in \mathcal{C}} \|F_j\|_{p_j}
\le C(4N)^{m-1}.
\end{equation}

On the support of the product of the functions $F_j\circ \Pi_j$, one has $x_1^0\ge {0}$  and $x_1^1\ge {1/2}$. Hence $K(\Pi x)$ is positive on this support. 
We restrict the integral to a domain where all functions $F_j$ take value $1$ and 
thus estimate the left-hand side of
\eqref{lastsillythingnov25} from below by  
$$\int_{[1/4,1/2]\times [-N,N]^{m-1}
\times [1,N]\times [-N,N]^{m-1}
}
\frac{x_1^0+x^1_1}{|x^0+x^1|^{m+1}}\,dx.
$$
For $2\le h\le m$, we introduce
variables $y_h^+=x_h^0+x_h^1$
and $y_h^-=x_h^0-x_h^1$ and combine these variables to  a vector $y=(y^+,y^-)$.
Shrinking the domain of integration from line to line as convenient, we estimate the last display from below by
$$
 2^{2-2m}\int_{[1/4,1/2]}\int_{[1,N]}
 \int_{[-N/2,N/2]^{2m-2}}
\frac{x_1^0+x^1_1}{|(x_1^0+x_1^1,y^+)|^{m+1}}\,dy \, dx_1^1 \, dx_1^0
$$
$$
\ge 2^{-2m}N^{m-1}\int_{[1,N]}\int_{ [-N/2,N/2]^{m-1}}
\frac{x^1_1}{|(2x_1^1,y^+)|^{m+1}}\,
 dy^+ \, dx_1^1
$$
$$
\ge 2^{-2m}N^{m-1}\int_{[1,N/4]}\int_{
 [x_1^1,2x_1^1]^{m-1}}
 \frac{x^1_1}{(2mx_1^1)^{m+1}}\,dy^+ \, dx_1^1
$$
$$
\ge (8m)^{-m-1}N^{m-1} \int_{[1,N/4]}
 \frac{1}{x_1^1}\,dx_1^1=  (8m)^{-m-1}N^{m-1}\ln(N/4) .$$
The right-hand side grows faster in $N$ than the right-hand side of 
\eqref{lastsillythingnov25}, which gives  a contradiction for 
$N$ large enough. This completes the proof of Theorem \ref{T:sharp}.

\end{document}